\documentclass{article}

\usepackage{amsthm}
\usepackage{amsmath}
\usepackage{amsfonts}
\usepackage{paralist}
\usepackage{ifthen}
\usepackage{tikz}

\newtheorem{theorem}{Theorem}[section]
\newtheorem{lemma}[theorem]{Lemma}
\newtheorem{proposition}[theorem]{Proposition}
\newtheorem{corollary}[theorem]{Corollary}

\newcommand{\parder}[3][Default]{
	\frac{\partial \ifthenelse{\equal{#1}{Default}}{}{^{#1}}#2}{
              \partial #3 \ifthenelse{\equal{#1}{Default}}{}{^{#1}}}}

\numberwithin{equation}{section}

\makeatletter
\newcommand{\nolisttopbreak}{\vspace{\topsep}\nobreak\@afterheading}
\makeatother

\newcommand{\jac}{{\mathcal{J}}}
\newcommand{\hess}{{\mathcal{H}}}
\newcommand{\ie}{\operatorname{IE}}
\newcommand{\pe}{\operatorname{PE}}
\newcommand{\N}{\mathbb{N}}
\newcommand{\Z}{\mathbb{Z}}
\newcommand{\Mat}{\operatorname{Mat}}
\newcommand{\GL}{\operatorname{GL}}
\newcommand{\rk}{\operatorname{rk}}
\newcommand{\tr}{\operatorname{tr}}
\newcommand{\trdeg}{\operatorname{trdeg}}
\newcommand{\chr}{\operatorname{char}}
\newcommand{\tp}{^{\rm t}}
\newcommand{\sdots}{\vdots\,\vdots\,\vdots}

\newcommand{\zeromat}{\mbox{\begin{tikzpicture} \useasboundingbox (-0.1,-0.125) -- (0.1,0.125);
\draw (-0.05,0.05) -- node[anchor=center]{\rm 0} (0.05,-0.05); \end{tikzpicture}}}

\begin{document}

\title{Quadratic homogeneous polynomial maps $H$ and Keller maps $x+H$ with 
$3 \le \rk \jac H \le 4$}

\author{Michiel de Bondt}

\maketitle

\begin{abstract}
We compute by hand all quadratic homogeneous polynomial maps $H$ and all 
Keller maps of the form $x + H$, for which $\rk \jac H = 3$, over a field 
of arbitrary characteristic.

Furthermore, we use computer support to compute Keller maps of the form 
$x + H$ with $\rk \jac H = 4$, namely:
\begin{compactitem}

\item all such maps in dimension $5$ over fields with $\frac12$;

\item all such maps in dimension $6$ over fields without $\frac12$.

\end{compactitem}
We use these results to prove the following over fields of arbitrary 
characteristic: for Keller maps $x + H$ for which $\rk \jac H \le 4$, 
the rows of $\jac H$ are dependent over the base field.
\end{abstract}

\section{Introduction}

Let $n$ be a positive integer and let $x = (x_1,x_2,\ldots,x_n)$ be an $n$-tuple
of variables. 
We write $a|_{b=c}$ for the result of substituting $b$ by $c$ in $a$.

Let $K$ be any field. In the scope of this introduction, denote
by $L$ an unspecified (but big enough) field, which contains $K$ or even $K(x)$.

For a polynomial or rational map $H = (H_1,H_2,\ldots,H_m) \in L^m$, write $\jac H$ or 
$\jac_x H$ for the Jacobian matrix of $H$ with respect to $x$. So
$$
\jac H = \jac_x H = \left(\begin{array}{cclc}
\parder{}{x_1} H_1 & \parder{}{x_2} H_1 & \cdots & \parder{}{x_n} H_1 \\
\parder{}{x_1} H_2 & \parder{}{x_2} H_2 & \cdots & \parder{}{x_n} H_2 \\
\vdots & \vdots & \sdots & \vdots \\
\parder{}{x_1} H_m & \parder{}{x_2} H_m & \cdots & \parder{}{x_n} H_m
\end{array}\right)
$$

Denote by $\rk M$ the rank of a matrix $M$, whose entries are contained in $L$,
and write $\trdeg_K L$ for the transcendence degree of $L$ over $K$. 
It is known that $\rk \jac H \le \trdeg_K K(H)$ for a rational map $H$ of any 
degree, with equality if $K(H) \subseteq K(x)$ is separable, in particular if
$K$ has characteristic zero. This is proved in \cite[Th.\@ 1.3]{1501.06046},
see also \cite[Ths.\@ 10, 13]{DBLP:conf/mfcs/PandeySS16}.

Let a \emph{Keller map} be a polynomial map $F \in K[x]^n$, for
which $\det \jac F \in K \setminus \{0\}$. 
If $H \in K[x]^n$ is homogeneous of degree at least $2$, then $x + H$ is
a Keller map, if and only if $\jac H$ is nilpotent. 

We say that a matrix $M \in \Mat_n(L)$ is \emph{similar over $K$}
to a matrix $\tilde{M} \in \Mat_n(L)$, if there exists a $T \in \GL_n(K)$ such that
$\tilde{M} = T^{-1}MT$. If $\jac H$ is similar over $K$ to a triangular 
matrix, say that $T^{-1}(\jac H)T$ is a triangular matrix, then
$$
\jac \big(T^{-1}H(Tx)\big) = T^{-1}(\jac H)|_{x=Tx}T
$$ 
is triangular as well.

Section $2$ is about quadratic homogeneous maps $H$ in general, with the
focus on compositions with invertible linear maps, and the invariant 
$r = \rk \jac H$. In section $3$, a classification is given for the 
case $r = 3$.

In section $4$, a classification is given for the case $r = 3$,
combined with the nilpotency of $\jac H$. Nilpotency is an invariant of 
conjugations with invertible linear maps, but is not an invariant of 
compositions with invertible linear maps in general. 

In section $5$, we compute all Keller maps $x + H$ with $H$ quadratic homogeneous,
for which $\rk \jac H = 4$ in dimension $5$ over fields with $\frac12$,
and in dimension $6$ over fields without $\frac12$.  
We use these results to prove the following over fields of arbitrary 
characteristic: for Keller maps $x + H$ with $H$ quadratic homogeneous,
for which $\rk \jac H \le 4$, the rows of $\jac H$ are dependent over the base field.

\section{rank {\mathversion{bold}$r$}}

\begin{theorem} \label{rkr}
Let $H \in K[x]^m$ be a quadratic homogeneous polynomial map, and
$r := \rk \jac H$. 

Then there are $S \in \GL_m(K)$ and $T \in \GL_n(K)$, such that
for $\tilde{H} := S H(Tx)$, only the first $\frac12 r^2 + \frac12 r$ rows
of $\jac \tilde{H}$ may be nonzero, and one of the following statements holds:
\begin{enumerate}[\upshape(1)]

\item Only the first $\frac12 r^2 - \frac12 r + 1$ rows of $\jac \tilde{H}$ may be nonzero;

\item $\chr K \ne 2$ and only the first $r$ columns of $\jac \tilde{H}$ are nonzero;

\item $\chr K = 2$ and only the first $r+1$ columns of $\jac \tilde{H}$ are nonzero.

\end{enumerate}
Conversely, $\rk \jac \tilde{H} \le r$ if either $\tilde{H}$ is as in 
{\upshape(2)} or {\upshape(3)}, or $1 \le r \le 2$ and $\tilde{H}$ is as in
{\upshape(1)}.
\end{theorem}

\begin{corollary} \label{rk4}
Let $H \in K[x]^m$ be a quadratic homogeneous polynomial map. 
Suppose that $r := \rk \jac H \le 4$.

Then there are $S \in \GL_m(K)$ and $T \in \GL_n(K)$, such that
for $\tilde{H} := S H(Tx)$, only the first $\frac12 r^2 + \frac12 r$
of $\jac \tilde{H}$ may be nonzero, and one of the following statements holds:
\begin{enumerate}[\upshape (1)]

\item Only the first $r+1$ rows of $\jac \tilde{H}$ may be nonzero;

\item $r = 4$, $\tilde{H}_i \in K[x_1,x_2,x_3]$ for all $i \ge 2$, and 
$\chr K \neq 2$;

\item $r = 4$, $\tilde{H}_i \in K[x_1,x_2,x_3,x_4,x_5^2,x_6^2,\ldots,x_n^2]$ 
for all $i \ge 2$, and $\chr K = 2$;

\item Only the first $r+1$ columns of $\jac \tilde{H}$ may be nonzero;

\item $r = 4$, only the leading principal minor matrix of size $6$ of
$\jac \tilde{H}$ is nonzero, and $\chr K = 2$.

\end{enumerate} 
\end{corollary}

\begin{lemma} \label{23}
If $\tilde{H}$ is as in {\upshape(2)} or {\upshape(3)} of 
theorem {\upshape\ref{rkr}}, then theorem {\upshape\ref{rkr}} 
holds for $H$. 
\end{lemma}

\begin{proof}
The number of terms of degree $d$ in $x_1, x_2, \ldots, x_r$ is
$$
\binom{r-1+d}{d}
$$
and the number of square-free terms of degree $d$ in $x_1, x_2, \ldots, x_r, x_{r+1}$
is
$$
\binom{r+1}{d}
$$
which is both $\frac12 r^2 + \frac12 r$ if $d = 2$.

If $\tilde{H}$ is as in (2) of theorem \ref{rkr}, then $\chr K \ne 2$
and the rows of $\jac \tilde{H}$ are dependent over $K$ on the $\frac12 r^2 + \frac12 r$ rows of
$$
\jac (x_1^2,x_1 x_2, \ldots, x_r^2)
$$
If $\tilde{H}$ is as in (3) of theorem \ref{rkr}, then $\chr K = 2$
and the rows of $\jac \tilde{H}$ are dependent over $K$ on the $\frac12 r^2 + \frac12 r$ rows of
$$
\jac (x_1 x_2, x_1 x_3, \ldots, x_r x_{r+1})
$$
So at most $\frac12 r^2 + \frac12 r$ rows of $\jac \tilde{H}$ as in theorem \ref{rkr}
need to be nonzero.

If $\tilde{H}$ is as in (2) of theorem \ref{rkr}, then 
$\rk \jac \tilde{H} \le r$ is direct. If $\tilde{H}$ is as in
(3) of theorem \ref{rkr}, then $\rk \jac \tilde{H} \le r$ follows as well,
because $\jac \tilde{H} \cdot x = 0$ if $\chr K = 2$.
\end{proof}

\begin{lemma} \label{Irlem}
Let $M$ be a nonzero matrix whose entries are linear forms in $K[x]$. Suppose
that $r := \rk M$ does not exceed the cardinality of $K$. 

Then there are invertible matrices $S$ and $T$ over $K$, such that for $\tilde{M} := S M T$, 
$$
\tilde{M} = \tilde{M}^{(1)} L_1 + \tilde{M}^{(2)} L_2 + \cdots + \tilde{M}^{(n)} L_n
$$
where $\tilde{M}^{(i)}$ is a matrix with coefficients in $K$ for each $i$,
$L_1, L_2, \ldots, L_n$ are independent linear forms, and
$$
\tilde{M}^{(1)} = \left( \begin{array}{cc}
I_r & \zeromat \\ \zeromat & \zeromat \\                          
\end{array} \right)
$$
\end{lemma}

\begin{proof}
Since $\rk M = r \ge 1$, $M$ has a minor matrix of size $r \times r$ whose determinant 
is nonzero. Assume without loss of generality that the determinant of the leading 
principal minor matrix of size $r \times r$ of $M$ is nonzero. Then this determinant
$f$ is a homogeneous polynomial of degree $r$. From \cite[Lemma 5.1 (ii)]{1310.7843},
it follows that there exists a $v \in K^n$ such that $f(v) \ne 0$. 

Take independent linear forms $L_1, L_2, \ldots, L_n$ such that $L_i(v) = 0$
for all $i \ge 2$. Then $L_1(v) \ne 0$, and we may assume that $L_1(v) = 1$.
We can write
$$
M = M^{(1)} L_1 + M^{(2)} L_2 + \cdots + M^{(n)} L_n
$$
where $M^{(i)}$ is a matrix with coefficients in $K$ for each $i$.
If we substitute $x = v$ on both sides, we see that $\rk M^{(1)} \le r$ and that 
the leading principal minor matrix of size $r \times r$ of $M^{(1)}$ has a
nonzero determinant. 

So $\rk M^{(1)} = r$, and we can choose invertible matrices $S$ and $T$ over $K$,
such that
$$
S M^{(1)} T = \left( \begin{array}{cc}
I_r & \zeromat \\ \zeromat & \zeromat \\                          
\end{array} \right)
$$
So we can take $\tilde{M}^{(i)} = S M^{(i)} T$ for each $i$.
\end{proof}

Suppose that $\tilde{M}$ is as in lemma \ref{Irlem}. Write
\begin{equation} \label{ABCD}
\tilde{M} = \left( \begin{array}{cc} A & B \\ C & D \end{array} \right)
\end{equation}
where $A \in \Mat_r(K)$. If we extend $A$ with one row and one
column of $\jac H$, we get an element of $\Mat_{r+1}\big(K(x)\big)$ from which the
determinant is zero. If we focus on the coefficients of $L_1^r$ 
and $L_1^{r-1}$ of this determinant, we see that
\begin{equation} \label{D0CB0}
D = 0 \qquad \mbox{and} \qquad C \cdot B = 0
\end{equation}
respectively. In particular, 
\begin{equation} \label{rkCrkBler}
D = 0 \qquad \mbox{and} \qquad C \cdot B = 0
\end{equation}

\begin{lemma} \label{Ccoldep1}
Let $\tilde{H} \in K[x]^m$, such that $\jac \tilde{H}$ is as $\tilde{M}$ in 
lemma \ref{Irlem}. 

Suppose that either $\chr K \ne 2$ or the rows of $B$ are dependent over $K$. 
Then the following hold.
\begin{enumerate}[\upshape (i)]

\item The columns of $C$ in \eqref{ABCD} are dependent over $K$.

\item If $C \ne 0$, then there exists a $v \in K^n$ of which the first $r$
coordinates are not all zero, such that
$$
(\jac \tilde{H}) \cdot v = \left( \begin{array}{cc} I_r & \zeromat \\ 
\zeromat & \zeromat \end{array} \right) \cdot x
$$

\end{enumerate}
\end{lemma}

\begin{proof}
The claims are direct if $C = 0$, so assume that $C \ne 0$. Then there exists
an $i > r$, such that $\tilde{H}_i \ne 0$.
\begin{enumerate}[\upshape (i)]

\item Take $v$ as in (ii). From $D = 0$, we deduce that $C \cdot v' = (C | D) \cdot v = 0$ 
for some nonzero $v' \in K^r$, which yields (i).

\item Take $v$ as in lemma \ref{Irlem}, and write $v = (v',v'')$,
such that $v' \in K^r$ and $v'' \in K^{n-r}$. Since $\tilde{H}$ is quadratic
homogeneous, we have
$$
(\jac \tilde{H}) \cdot v = (\jac \tilde{H})|_{x=v} \cdot x = \tilde{M}^{(1)} \cdot x
= \left( \begin{array}{cc} I_r & \zeromat \\ \zeromat & \zeromat \end{array} \right) \cdot x
$$
So it remains to show that $v' \ne 0$. We distinguish two cases:
\begin{itemize}

\item \emph{the rows of $B$ are dependent over $K$.} 

Take $w \in K^r$ nonzero, such that $w\tp B = 0$. Then  
$$
w\tp A\,v' = w\tp A\,v' + w\tp B\,v'' = w\tp (A|B)\,v 
= w\tp \left( \begin{smallmatrix} x_1 \\ x_2 \\[-5pt] \vdots \\ x_r \end{smallmatrix} \right)
= \sum_{i=1}^r w_i x_i \ne 0
$$
so $v' \ne 0$.

\item \emph{$\chr K \ne 2$.} 

From $CB = 0$, we deduce that
$$
CA\,v' = CA\,v' + CB\,v'' = C\,(A|B)\,v
= C \left( \begin{smallmatrix} x_1 \\ x_2 \\[-5pt] \vdots \\ x_r \end{smallmatrix} \right)
= 2 \left( \begin{smallmatrix} \tilde{H}_{r+1} \\ \tilde{H}_{r+2} \\[-5pt] \vdots \\ 
    \tilde{H}_{m} \end{smallmatrix} \right)
$$
As $\tilde{H}_{i} \ne 0$ for some $i > r$, the right-hand side is nonzero, so $v' \ne 0$. \qedhere

\end{itemize}
\end{enumerate}
\end{proof}

\begin{lemma} \label{BCrkr}
Let $\tilde{H} \in K[x]^m$, such that $\jac \tilde{H}$ is as $\tilde{M}$ in 
lemma \ref{Irlem}. 

Suppose that $\rk B + \rk C = r$ and that the columns of $C$ are dependent
over $K$. Then the column space of $B$ contains a nonzero constant vector.
\end{lemma}

\begin{proof}
From $\rk C + \rk B = r$ and $CB = 0$, we deduce that $\ker C$ is equal 
to the column space of $B$. Hence any $v' \in K^r$ 
such that $C v' = 0$ is contained in the column space of $B$.
\end{proof}

\begin{proof}[Proof of theorem \ref{rkr}]
From lemma \ref{Fqlem} below, it follows that we may assume that $K$ has at least
$r$ elements. Let $M = \jac H$ and take $S$ and $T$ as in lemma \ref{Irlem}.
Then $S (\jac H) T$ is as $\tilde{M}$ in lemma \ref{Irlem}. Let $\tilde{H} := SH(Tx)$.
Then $\jac \tilde{H} = S (\jac H) |_{x=Tx} T$ is as $\tilde{M}$ in lemma \ref{Irlem}
as well, but for different linear forms $L_i$. 

Take $\tilde{M} = \jac \tilde{H}$ and take $A$, $B$, $C$, $D$ as in \eqref{ABCD}.
We distinguish four cases:
\begin{itemize}

\item \emph{The column space of $B$ contains a nonzero constant vector.}

Then there exists an $U \in \GL_m(K)$, such that the column space of
$U \tilde{M}$ contains $e_1$, because $D = 0$. Consequently, the matrix which 
consists of the last $m-1$ rows of $\jac (U \tilde{H}) = U \tilde{M}$ has rank $r-1$.
By induction on $r$, it follows that we can choose $U$ such that
only 
$$
\tfrac12(r-1)^2 + \tfrac12(r-1) = (\tfrac12 r^2 - r + \tfrac12) + 
(\tfrac12 r - \tfrac12) = \tfrac12 r^2 - \tfrac12 r
$$
rows of $\jac (U \tilde{H})$ are nonzero besides the first row of 
$\jac (U \tilde{H})$. So $U \tilde{H}$ is as $\tilde{H}$ in (1) of 
theorem \ref{rkr}. 

\item \emph{The rows of $B$ are dependent over $K$ in pairs.}

If $B \ne 0$, then the column space of $B$ contains a nonzero 
constant vector, and the case above applies because $D = 0$.

So assume that $B = 0$. Then only the first $r$ columns of 
$\jac \tilde{H}$ may be nonzero, because $D = 0$. Since
$\rk \jac \tilde{H} = r$, the first $r$ columns of $\jac \tilde{H}$ 
are indeed nonzero.
Furthermore, it follows from $\jac \tilde{H} \cdot x = 2 \tilde{H}$
that $\chr K \neq 2$. So $\tilde{H}$ is as in (2) of theorem \ref{rkr},
and the result follows from lemma \ref{23}.

\item \emph{$\chr K = 2$ and $\rk B \le 1$.}

If the rows of $B$ are dependent over $K$ in pairs, then the second case
above applies, so assume that the rows of $B$ are not dependent over $K$ 
in pairs. 

On account of \cite[Theorem 2.1]{1601.00579}, 
the columns of $B$ are dependent over $K$ in pairs. As $D = 0$, there exists an 
$U'' \in \GL_{n-r}(K)$ such that only the first column
of
$$
\binom{B}{D} U''
$$
may be nonzero. Hence there exists an $U \in \GL_n(K)$ 
such that only the first $r+1$ columns of $(\jac \tilde{H})\, U$ 
may be nonzero. Consequently, $\tilde{H}(Ux)$ is as 
$\tilde{H}$ in (3) of theorem \ref{rkr}, and the result follows from lemma \ref{23}.

\item \emph{None of the above.}

We first show that $\rk C \le r - 2$. So assume that $\rk C \ge r - 1$.
From $\rk C + \rk B \le r$, it follows that $\rk B \le 1$. As the last case
above does not apply, $\chr K \ne 2$. From (i) of lemma \ref{Ccoldep1}, it follows 
that the columns of $C$ are dependent over $K$. As the first case above does not
apply, it follows from lemma \ref{BCrkr} that $\rk C + \rk B < r$.
So $\rk B = 0$ and the rows of $B$ are dependent over $K$ in pairs,
which is the second case above, and a contradiction.
So $\rk C \le r - 2$ indeed.

By induction on $r$, it follows that $C$ needs to have at most 
$$
\tfrac12 (r-2)^2 + \tfrac12 (r-2) = (\tfrac12 r^2 - 2r + 2) + (\tfrac12 r - 1)
= \tfrac12 r^2 - \tfrac32 r + 1
$$
nonzero rows. As $A$ has $r$ rows, there exists an $U \in \GL_m(K)$ such that
$U\tilde{H}$ is as $\tilde{H}$ in (1) of theorem \ref{rkr}.

\end{itemize}
The last claim of theorem \ref{rkr} follows from lemma \ref{23}
and the fact that $\frac12 r^2 - \frac12 r + 1 = r$ if $1 \le r \le 2$.
\end{proof}

\begin{lemma} \label{Fqlem}
Let $L$ be an extension field of $K$. If theorem {\upshape\ref{rkr}} 
holds for $L$ instead of $K$, then theorem {\upshape\ref{rkr}} holds.
\end{lemma}

\begin{proof}
We only prove lemma \ref{Fqlem} for the first claim of theorem \ref{rkr},
because the second claim can be treated in a similar manner, and the last
claim does not depend on the actual base field.

Suppose $H$ satisfies the first claim of theorem \ref{rkr}, 
but with $L$ instead of $K$. If $m \le \frac12 r^2 + \frac12 r$, then 
only the first $\frac12 r^2 + \frac12 r$ rows of $\jac H$ may be nonzero,
and $H$ satisfies the first claim of theorem \ref{rkr}.

So assume that $m > \frac12 r^2 + \frac12 r$. Then the rows of $\jac H$ are 
dependent over $L$. Since $L$ is a vector space over $K$, 
the rows of $\jac H$ are dependent over $K$. So we
may assume that the last row of $\jac H$ is zero. By induction on $m$,
$(H_1,H_2,\ldots,H_{m-1})$ satisfies the first claim for $H$ in
in theorem \ref{rkr}. As $H_m = 0$,
we conclude that $H$ satisfies the first claim of theorem \ref{rkr}.
\end{proof}

\begin{proof}[Proof of corollary \ref{rk4}]
If (2) or (3) of theorem \ref{rkr} applies, then (4) of corollary \ref{rk4} 
follows. So assume that (1) of theorem \ref{rkr} applies. Then only the 
first $\tfrac12 r^2 - \tfrac12 r + 1$ rows of $\jac \tilde{H}$ may be nonzero.

Suppose first that $r \le 3$. Then 
$$
\tfrac12 r^2 - \tfrac12 r + 1 \le \tfrac32 r - \tfrac12 r + 1 = r + 1
$$
and corollary \ref{rk4} follows.

Suppose next that $r = 4$. 
Take $\tilde{M} = \jac \tilde{H}$ and take $A$, $B$, $C$, $D$ as in \eqref{ABCD}.
We distinguish three cases.
\begin{itemize}

\item \emph{The column space of $B$ contains a nonzero constant vector.}

Then there exists an $U \in \GL_m(K)$, such that the column space of
$U \tilde{M}$ contains $e_1$. So the matrix which consists of the last 
$m-1$ rows of $\jac (U \tilde{H}) = U \tilde{M}$ has rank $r-1$.

Make $\tilde{U}$ from $U$ by replacing its first row by the zero row.
Then $\rk \jac (\tilde{U} \tilde{H}) = r - 1$, and we can apply theorem 
\ref{rkr} to $\tilde{U} \tilde{H}$. 
\begin{compactitem}

\item If case (1) of theorem \ref{rkr} applies for $\tilde{U} \tilde{H}$,
then case (1) of corollary \ref{rk4} follows, because
$$
\tfrac12 (r-1)^2 - \tfrac12 (r-1) + 1 = \tfrac92 - \tfrac32 + 1 = 4 = r
$$

\item If case (2) of theorem \ref{rkr} applies for $\tilde{U} \tilde{H}$, 
then case (2) of corollary \ref{rk4} follows.

\item If case (3) of theorem \ref{rkr} applies for $\tilde{U} \tilde{H}$, 
then case (3) of corollary \ref{rk4} follows.

\end{compactitem}

\item \emph{$\rk B \le 1$.}

If the columns of $B$ are dependent over $K$ in pairs, then (4) of corollary 
\ref{rk4} is satisfied. So assume that the columns of $B$ are not dependent over 
$K$ in pairs. Then $B \ne 0$, and from \cite[Theorem 2.1]{1601.00579}, 
it follows that the rows of $B$ are dependent over $K$ in pairs. 
Hence the first case above applies.

\item \emph{$\rk C \le 1$.}

Then it follows from theorem \ref{rkr} that at most one row of $C$ needs to
be nonzero. So (1) of corollary \ref{rk4} is satisfied.

\item \emph{None of the above.}

We first show that $\rk C = \rk B = 2$ and that the columns
of $C$ are independent over $K$.
Since $\rk C + \rk B \le r = 4$, we deduce from $\rk B > 1$ and $\rk C > 1$
that $\rk C = \rk B = 2$. So $\rk C + \rk B = 4 = r$. As the first case above 
does not apply, it follows from lemma \ref{BCrkr} that the columns
of $C$ are independent over $K$.

From (i) of lemma \ref{Ccoldep1}, we deduce that
$\chr K = 2$ and that the rows of $B$ are independent over $K$.
Since the $\rk C + 2$ columns of $C$ are independent over $K$, it follows from
theorem \ref{rkr} that $C$ needs to have at most 
$$
\tfrac12 \cdot 2^2 - \tfrac12 \cdot 2 + 1 = 2 - 1 + 1 = 2
$$
nonzero rows. 
Since the $\rk B + 2$ rows of $B$ are independent over $K$, it follows from
\cite[Theorem 2.3]{1601.00579} that $B$ needs to have at most 
$2$ nonzero columns, because the first case above does not apply. 
So (5) of corollary \ref{rk4} is satisfied.
\qedhere

\end{itemize}
\end{proof}

The last case in corollary \ref{rk4} is indeed necessary, 
e.g.\@ $\jac \tilde{H} = \hess (x_1 x_2 x_3 + x_4 x_5 x_6)$, or
$\jac \tilde{H} = \hess (x_1 x_2 x_3 + x_1 x_5 x_6 + x_4 x_2 x_6 + x_4 x_5 x_3)$.

\section{rank 3}

\begin{theorem} \label{rk3}
Let $H \in K[x]^m$ be a quadratic homogeneous polynomial map, such that 
$r := \rk \jac H = 3$. Then we can choose $S \in \GL_m(K)$ and $T \in \GL_n(K)$, 
such that for $\tilde{H} := S H(Tx)$, one of the following statements holds:
\begin{enumerate}[\upshape(1)]

\item Only the first $3$ rows of $\jac \tilde{H}$ may be nonzero;

\item Only the first $4$ rows of $\jac \tilde{H}$ may be nonzero, and
$$
(\tilde{H}_1,\tilde{H}_2,\tilde{H}_3,\tilde{H}_4) = 
(\tilde{H}_1, \tfrac12 x_1^2, x_1 x_2,  \tfrac12 x_2^2)
$$
(in particular, $\chr K \ne 2$);

\item Only the first $4$ rows of $\jac \tilde{H}$ may be nonzero,
$$
\jac (\tilde{H}_1,\tilde{H}_2,\tilde{H}_3,\tilde{H}_4) = 
\jac (\tilde{H}_1, x_1 x_2, x_1 x_3, x_2 x_3)
$$
and $\chr K = 2$;

\item $\tilde{H}$ is as in {\upshape(2)} or {\upshape(3)} of theorem 
{\upshape\ref{rkr}};

\item Only the first $4$ rows of $\jac \tilde{H}$ may be nonzero, and
$$
\big(\tilde{H}_1,\tilde{H}_2,\tilde{H}_3,\tilde{H}_4\big) = 
\big( x_1 x_3 + c x_2 x_4, x_2 x_3 - x_1 x_4, 
\tfrac12 x_3^2 + \tfrac{c}2 x_4^2, \tfrac12 x_1^2 + \tfrac{c}2 x_2^2 \big)
$$
for some nonzero $c \in K$ (in particular, $\chr K \ne 2$).

\end{enumerate}
Conversely, $\rk \jac \tilde{H} \le 3$ in each of the five statements above.
\end{theorem}

\begin{corollary} \label{rktrdeg}
Let $H \in K[x]^m$ be a quadratic homogeneous polynomial map, such that
$\rk \jac H \le 3$. If $\chr K \neq 2$, then $\rk \jac H = \trdeg_K K(H)$. 
\end{corollary}

\begin{proof}
Since $\rk \jac H \le \trdeg_K K(H)$, it suffices to show that 
$\trdeg_K K(H) \le 3$ if $\chr K \neq 2$.
In (5) of theorem \ref{rk3}, we have $\trdeg_K K(H) \le 3$ because
$$
\tilde{H}_1^2 + c \tilde{H}_2^2 - 4 \tilde{H}_3 \tilde{H}_4 = 0
$$
In the other cases of theorem \ref{rk3} where $\chr K \neq 2$,
$\trdeg_K K(H) \le 3$ follows directly.
\end{proof}

\begin{lemma} \label{Ccoldep2}
Let $\tilde{H} \in K[x]^m$, such that $\jac \tilde{H}$ is as $\tilde{M}$ in 
lemma \ref{Irlem}. 

If $\rk C = 1$ in \eqref{ABCD} and $r$ is odd, then the columns of $C$ in 
\eqref{ABCD} are dependent over $K$.
\end{lemma}

\begin{proof}
The case where $\chr K \ne 2$ follows from (i) of lemma \ref{Ccoldep1}, so
assume that $\chr K = 2$. 
Since $\rk C = 1 = \frac12 \cdot 1^2 + \frac12 \cdot 1$, we deduce from 
theorem \ref{rkr} that the rows of $C$ are dependent over $K$ in pairs.
Say that the first row of $C$ is nonzero. 

As $r$ is odd, it follows from proposition \ref{evenrk} below that
$\rk \hess \tilde{H}_{r+1} < r$. Hence there exists a $v' \in K^r$ 
such that $(\hess \tilde{H}_{r+1}) \, v' = 0$, and
$$
(\jac \tilde{H}_{r+1}) \, v' = x\tp (\hess \tilde{H}_{r+1}) \, = 0
$$
The row space of $C$ is spanned by $\jac \tilde{H}_{r+1}$,
so $C\,v' = 0$.
\end{proof}

\begin{proposition} \label{evenrk}
Let $M \in \Mat_{n,n}(K)$ be either a symmetric matrix, or an anti-symmetric 
matrix with zeroes on the diagonal.

Then there exists a lower triangular matrix $T \in \Mat_{n,n}(K)$ with 
ones on the diagonal, such that $T\tp M T$ is the product of a symmetric 
permutation matrix and a diagonal matrix.

In particular, $\rk M$ is even if $M$ is an anti-symmetric 
matrix with zeroes on the diagonal.
\end{proposition}

\begin{proof}
We describe an algorithm to transform $M$ to the product of a symmetric 
permutation matrix and a diagonal matrix. We distinguish three cases.
\begin{itemize}

\item \emph{The last column of $M$ is zero.} 

Advance with the principal minor of $M$ that we get by removing row and 
column $n$.

\item \emph{The entry in the lower right corner of $M$ is nonzero.}

Use $M_{nn}$ as a pivot to clean the rest of the last column and the last row
of $M$. Advance with the principal minor of $M$ that we get by removing row 
and column $n$.

\item \emph{None of the above.}

Let $i$ be the index of the lowest nonzero entry in the last column of $M$.
Use $M_{in}$ and $M_{ni}$ as pivots to clean the rest of columns $i$ and $n$
of $M$ and rows $i$ and $n$ of $M$. Advance with the principal minor of $M$ 
that we get by removing rows and columns $i$ and $n$.

\end{itemize}
If $M$ is $M$ is an anti-symmetric matrix with zeroes on the diagonal, then so is 
$T\tp M T$. By combining this with that $T\tp M T$ is the product of a symmetric 
permutation matrix and a diagonal matrix, we infer the last claim.
\end{proof}

Notice that over characteristic $2$, any Hessian
matrix of a polynomial is \mbox{(anti-)}\allowbreak symmetric with zeroes on the diagonal, 
which has even rank by extension of scalars. Furthermore, one can show
that any nondegenerate quadratic form in odd dimension $r$ over a perfect field
of characteristic $2$ is equivalent to 
$$
x_1 x_2 + x_3 x_4 + \cdots + x_{r-2} x_{r-1} + x_r^2
$$

\begin{corollary} \label{symdiag}
Let $\chr K \neq 2$ and $M \in \Mat_{n,n}(K)$ be a symmetric matrix.

Then there exists a $T \in \GL_{n}(K)$, such that $T\tp M T$ is 
a diagonal matrix.
\end{corollary}

\begin{proof}
Notice that the permutation matrix $P$ in proposition \ref{evenrk}
only has cycles of length $1$ and $2$, because just like $M$,
the product of $P$ and the diagonal matrix $D$
is symmetric. Furthermore, for every cycle of length $2$, the
entries on the diagonal of $D$ which correspond to the two 
coordinates of that cycle are equal. Since
$$
\left( \begin{array}{cc} 1 & 1 \\ -1 & 1 \end{array} \right)
\left( \begin{array}{cc} 0 & c \\ c & 0 \end{array} \right)
\left( \begin{array}{cc} 1 & -1 \\ 1 & 1 \end{array} \right) =
\left( \begin{array}{cc} 2c & 0 \\ 0 & -2c \end{array} \right)
$$
we can get rid of the cycles of length $2$ in $P$.
\end{proof}

\begin{lemma} \label{rk3trafo}
Suppose that $H \in K[x_1,x_2,x_3,x_4]^4$, such that
$$
\jac H_4 = (\, x_1 ~ c x_2 ~ 0 ~ 0 \,)
\qquad \mbox{and} \qquad
\jac H \cdot v = \left( \begin{smallmatrix}
x_1 \\ x_2 \\ x_3 \\ 0 \end{smallmatrix} \right)
$$
for some nonzero $c \in K$, and a $v \in K^4$ of which
the first $3$ coordinates are not all zero. 

Suppose in addition that $\det \jac H = 0$ and that the last column of 
$\jac H$ does not generate a nonzero constant vector.

Then there are $S, T \in \GL_4(K)$, such that for
$\tilde{H} := S H(Tx)$, $\jac \tilde{H}$ is of the form
$$
\jac \tilde{H} =
\left( \begin{array}{cccc}
A_{11} & A_{12} & x_1 & c x_2 \\
A_{21} & A_{22} & x_2 & -x_1 \\
A_{31} & A_{32} & x_3 & \tilde{c}x_4 \\ 
x_1 & c x_2 & 0 & 0
\end{array} \right)
$$
where the coefficients of $x_1$ in $A_{11}$, $A_{21}$, $A_{21}$ are zero.
\end{lemma}

\begin{proof}
Since the last row of $\jac H$ is $(\, x_1 ~ c x_2  ~ 0 ~ 0 \,)$
and the last coordinate of $\jac \tilde{H} \cdot v$ is zero, we deduce that
$v_1 = v_2 = 0$. As the first $3$ coordinates of $v$ are not all zero, 
we conclude that $v_3 \ne 0$. 

Make $S$ from $v_3 I_4$ by changing column $4$ to $e_4$
Make $T$ from $I_4$ by changing column $3$ to $v_3^{-1} v$, i.e.\@
replacing the last entry of column $3$ by $v_3^{-1} v_4$, and by replacing 
the last column by an arbitrary scalar multiple of $e_4$. Then
\begin{align*}
(\jac \tilde{H}) \cdot e_3 
&= S\, (\jac H)|_{x=Tx} \cdot T e_3
 = \big(S\,(\jac H) \cdot v_3^{-1} v\big)\big|_{x=Tx} \\
&= v_3^{-1} S\, \left.
  \left( \begin{smallmatrix} x_1 \\ x_2 \\ x_3 \\ 0 \end{smallmatrix} \right)
  \right|_{x=Tx}
 = \left( \begin{smallmatrix} x_1 \\ x_2 \\ x_3 \\ 0 \end{smallmatrix} \right)
\end{align*}
and 
$$
\tilde{H}_4 = H_4(Tx) = H_4
$$
So the third column of $\jac \tilde{H}$ is as claimed. It follows that
$\jac \tilde{H}$ is as $\tilde{M}$ in lemma \ref{Irlem}, with 
$L_1 = x_3$, $L_2 = x_2$, $L_3 = x_1$, and $L_4 = x_4$. Define 
$A$, $B$, $C$ and $D$ as in \eqref{ABCD}.

Just like the last column of $\jac H$, the last column of $\jac \tilde{H}$ does 
not generate a nonzero constant vector. Consequently, $B_{11}$ and $B_{21}$ are not both 
zero. From \eqref{D0CB0}, it follows that $C_{11} B_{11} + C_{12} B_{21} = 0$.
So we can choose the last column of $T$, such that $B_{11} = C_{12} = c x_2$ 
and $B_{21} = -C_{11} = -x_1$.

The coefficient of $x_3$ in $B_{31}$ is zero, and by changing the third row 
of $I_4$ on the left of the diagonal in a proper way, we can get an 
$U \in \GL_4$ such that
$$
U \binom{B}{D} = \left( \begin{smallmatrix} 
B_{11} \\ B_{21} \\ \tilde{c} x_4 \\ D_{11} 
\end{smallmatrix} \right) = \left( \begin{smallmatrix} 
c x_2 \\ - x_3 \\ \tilde{c} x_4 \\ 0
\end{smallmatrix} \right)
$$
for some $\tilde{c} \in K$. Since $U^{-1}$ can be obtained 
by changing the third row of $I_4$ on the left of the diagonal 
in a proper way as well, we infer that
\begin{align*}
\jac \big(U^{-1} \tilde{H} (Ux)\big) \cdot e_3 
&= U^{-1}\, (\jac \tilde{H})|_{x = Ux} \,U \cdot e_3 
 = U^{-1}\, (\jac \tilde{H})|_{x = Ux} \cdot e_3 \\
&= U^{-1} \left.
  \left( \begin{smallmatrix} x_1 \\ x_2 \\ x_3 \\ 0 \end{smallmatrix} \right)
  \right|_{x=Ux}
 = \left( \begin{smallmatrix} x_1 \\ x_2 \\ x_3 \\ 0 \end{smallmatrix} \right)
\end{align*}
and 
$$
(U^{-1})_4\, \tilde{H} (Ux) = \tilde{H}_4 (Ux) = \tilde{H}_4
$$
So if we replace $S$ and $T$ by $U^{-1} S$ and $TU$ respectively, then $\tilde{H}$ will
be replaced by $U^{-1} \tilde{H}(Ux)$.

So we can get $B_{31}$ of the form $\tilde{c}x_4$ for some $\tilde{c} \in K$.
Finally, we can clean the coefficients of $x_1$ in $A_{11}$, $A_{21}$, $A_{21}$
with row operations, using $C_{11}$ as a pivot. So we can get the coefficients of 
$x_1$ in $A_{11}$, $A_{21}$, $A_{21}$ equal to zero.
\end{proof}

\begin{lemma} \label{rk3calc}
Suppose that $\det \jac \tilde{H} = 0$ and $\jac \tilde{H}$ is of the form 
$$
\left( \begin{array}{cccc}
A_{11} & A_{12} & x_1 & c x_2 \\
A_{21} & A_{22} & x_2 & -x_1 \\
A_{31} & A_{32} & x_3 & \tilde{c}x_4 \\ 
x_1 & c x_2 & 0 & 0
\end{array} \right)
$$
where $c,\tilde{c} \in K$, such that $c \neq 0$. If the coefficients of
$x_1$ in $A_{11}$, $A_{21}$, $A_{21}$ are zero, then 
$$
\tilde{H} = \left( \begin{array}{c}
x_1 x_3 + c x_2 x_4 \\ x_2 x_3 - x_1 x_4 \\
\frac12 x_3^2 + \frac{c}2 x_4^2 \\ \frac12 x_1^2 + \frac{c}2 x_2^2 
\end{array} \right)
\qquad \mbox{and} \qquad
\jac \tilde{H} = \left( \begin{array}{cccc}
x_3 & c x_4 & x_1 & c x_2 \\
- x_4 & x_3 & x_2 & -x_1 \\
0 & 0 & x_3 & c x_4 \\
x_1 & c x_2 & 0 & 0
\end{array} \right)
$$
\end{lemma}

\begin{proof}
Let $a_i$ and $b_i$ be the coefficients of $x_1$ and $x_2$ in
$A_{i2}$ respectively, for each $i \le 3$. Then
\begin{gather*}
\tilde{H} = \left( \begin{array}{c} 
a_1 x_1 x_2 + \frac12 b_1 x_2^2 + x_1 x_3 + c x_2 x_4 \\
a_2 x_1 x_2 + \frac12 b_2 x_2^2 + x_2 x_3 - x_1 x_4 \\
a_3 x_1 x_2 + \frac12 b_3 x_2^2 + \frac12 x_3^2 + \frac{\tilde{c}}2 x_4^2 \\
\frac12 x_1^2 + \frac{c}2 x_2^2 
\end{array} \right)
\qquad \mbox{and} \\
\jac \tilde{H} = \left( \begin{array}{cccc}
a_1 x_2 + x_3 & a_1 x_1 + b_1 x_2 + c x_4 & x_1 & c x_2 \\
a_2 x_2 - x_4 & a_2 x_1 + b_2 x_2 + x_3 & x_2 & -x_1 \\
a_3 x_2 & a_3 x_1 + b_3 x_2 & x_3 & \tilde{c} x_4 \\
x_1 & c x_2 & 0 & 0
\end{array} \right)
\end{gather*}
Consequently, it suffices to show that $a_i = b_i = 0$ for each $i \le 3$,
and that $\tilde{c} = c$.

Suppose that the coefficients of $x_1^4$ in $\det \jac \tilde{H}$ are zero.
Then we see by expansion along rows $3$, $4$, $1$, in that order,
that $a_3 x_1 = 0$. Hence the third row of $\jac \tilde{H}$ reads
$$
\jac \tilde{H}_3 = (\, 0 ~ b_3 x_2 ~ x_3 ~ \tilde{c} x_4 \,)
$$
Since the coefficients of $x_1^3 x_2$ and $x_1^3 x_3$ in $\det \jac \tilde{H}$ 
are zero, we see by expansion along rows $3$, $4$, $1$, in that order, that
$b_3 x_2 = a_1 x_1 = 0$. Hence the third row of $\jac \tilde{H}$ reads
$$
\jac \tilde{H}_3 = (\, 0 ~ 0 ~ x_3 ~ \tilde{c} x_4 \,)
$$
Since the coefficient of $x_2^3 x_3$ in $\det \jac \tilde{H}$ 
is zero, we see by expansion along rows $3$, $4$, $2$, in that order, that
$a_2 x_2 = 0$. So
$$
\jac \tilde{H} = \left( \begin{array}{cccc}
x_3 & b_1 x_2 + c x_4 & x_1 & c x_2 \\
- x_4 & b_2 x_2 + x_3 & x_2 & -x_1 \\
0 & 0 & x_3 & \tilde{c} x_4 \\
x_1 & c x_2 & 0 & 0
\end{array} \right)
$$
Since the coefficient of $x_1^2 x_3 x_4$ in $\det \jac \tilde{H}$ 
is zero, we see by expansion along row $3$, and columns $2$ and $1$, 
in that order, that $\tilde{c} x_4 = c x_4$.

Since the coefficient of $x_1 x_2^2 x_3$ in $\det \jac \tilde{H}$ 
is zero, we see by expansion along row $3$, and columns $1$, $4$, in that order, 
that $b_2 x_2 = 0$. Using that and that the coefficient of $x_1^2 x_2 x_3$ in 
$\det \jac \tilde{H}$ is zero, we see by expansion along row $3$, and columns 
$1$, $2$, in that order, that $b_1 x_2 = 0$. So $\tilde{H}$ is as claimed.
\end{proof}

\begin{proof}[Proof of theorem \ref{rk3}]
Take $\tilde{M} = \jac \tilde{H}$ and take $A$, $B$, $C$, $D$ as in \eqref{ABCD}.
We distinguish three cases:
\begin{itemize}

\item \emph{The column space of $B$ contains a nonzero constant vector.}

Take $U$ and $\tilde{U}$ as in the corresponding case in the proof of corollary 
\ref{rk4}.
\begin{compactitem}

\item If case (1) of theorem \ref{rkr} applies for $\tilde{U} \tilde{H}$,
then case (1) of theorem \ref{rk3} follows, because 
$$
\tfrac12 (r-1)^2 - \tfrac12 (r-1) + 1 = \tfrac42 - \tfrac22 + 1 = 2 = r - 1
$$

\item If case (2) of theorem \ref{rkr} applies for $\tilde{U} \tilde{H}$, 
then case (1) or case (2) of theorem \ref{rk3} follows.

\item If case (3) of theorem \ref{rkr} applies for $\tilde{U} \tilde{H}$, 
then case (1) or case (3) of theorem \ref{rk3} follows.

\end{compactitem}

\item \emph{The columns of $B$ are dependent over $K$ in pairs.}

If $\chr K = 2$, then $\tilde{H}$ is as in (3) of theorem \ref{rkr}, which
is included in (4) of theorem \ref{rk3}. So assume that $\chr K \neq 2$.
If $B = 0$, then $\tilde{H}$ is as in (2) of theorem \ref{rkr}, which
is included in (4) of theorem \ref{rk3} as well. So assume that $B \ne 0$.

From (i) of lemma \ref{Ccoldep1}, it follows that the columns of $C$ are
dependent over $K$. If $\rk C \ge 2$, then $\rk C = 2$ because 
$\rk B + \rk C \le r = 3$ and $B \ne 0$, and $B$ contains a nonzero constant
vector because of lemma \ref{BCrkr}. If $\rk C = 0$, then (1) of theorem \ref{rk3}
follows. So assume that $\rk C = 1$.

Since $\rk C = 1 = \frac12 \cdot 1^2 + \frac12 \cdot 1$, we deduce from 
theorem \ref{rkr} that the rows of $C$ are dependent over $K$ in pairs.
So we can choose $S$ such that only the first row of $C$ is nonzero.

From corollary \ref{symdiag}, it follows that there exists an $U' \in \GL_r(K)$,
such that 
$$
\hess_{x_1,x_2,\ldots,x_r} \big(\tilde{H}_{r+1}(U'x)\big)
= (U')^{-1} (\hess \tilde{H}_{r+1})\, U'
$$
is a diagonal matrix. Furthermore, the first entry on the diagonal is nonzero 
because $\hess \tilde{H}_{r+1} \ne 0$, and the last entry on the diagonal is zero
because the columns of $C$ are dependent over $K$. 
By adapting $S$, we can obtain that the entries on the diagonal 
of $\hess_{x_1,x_2,\ldots,x_r} (\tilde{H}_{r+1}(U'x)$ are $1$, $c$ and $0$, in that
order. 

By adapting $S$ and $T$, we can replace $\tilde{H}$ by
$$
\left( \begin{array}{cc}
(U')^{-1} & \zeromat \\ \zeromat & I_{m-3}        
\end{array} \right) \tilde{H}\left(\left( \begin{array}{cc}
U' & \zeromat \\ \zeromat & I_{n-3} 
\end{array} \right) x \right)
$$ 
Then the first row of $C$ becomes $(\, x_1 ~ c x_2  ~ 0 \,)$.
We distinguish two cases.
\begin{compactitem}

\item $c = 0$.

Notice that
$$
\jac (\tilde{H}|_{x_1=1}) = (\jac \tilde{H})|_{x_1=1} \cdot (\jac (1,x_2,x_3,\ldots,x_m))
$$
so $\rk \jac (\tilde{H}|_{x_1=1}) = r - 1 = 2$, and we can apply \cite[Theorem 2.3]{1601.00579}.
\begin{compactitem}

\item In the case of \cite[Theorem 2.3]{1601.00579} (1), case (2) of theorem
\ref{rkr} follows, which yields (4) of theorem \ref{rk3}.

\item In the case of \cite[Theorem 2.3]{1601.00579} (2), case (1) of theorem 
\ref{rk3} follows, because any linear combination of the rows of $\jac \tilde{H}$
which is dependent on $e_1$, is linearly dependent on $C_1$.

\item In the case of \cite[Theorem 2.3]{1601.00579} (3), case (1) or case (2) 
of theorem \ref{rk3} follows.

\end{compactitem}

\item $c \ne 0$.

From (ii) of lemma \ref{Ccoldep1} and lemma \ref{rk3trafo}, it follows that we can choose
$S$ and $T$, such that the leading principal minor matrix of size $4$ 
of $\jac \tilde{H}$ is of the form 
$$
\left( \begin{array}{cccc}
A_{11} & A_{12} & x_1 & c x_2 \\
A_{21} & A_{22} & x_2 & -x_1 \\
A_{31} & A_{32} & x_3 & \tilde{c}x_4 \\ 
x_1 & c x_2 & 0 & 0
\end{array} \right)
$$
where the coefficients of $x_1$ in $A_{11}$, $A_{21}$, $A_{21}$ are zero.
From lemma \ref{rk3calc}, case (5) of theorem \ref{rk3} follows. 

\end{compactitem}

\item \emph{None of the above.}

We first show that $\rk B \ge 2$. For that purpose, assume that $\rk B \le 1$.
Since the columns of $B$ are not dependent over $K$ in pairs, we deduce that
$B \ne 0$, and it follows from \cite[Theorem 2.1]{1601.00579} that the rows of 
$B$ are dependent over $K$ in pairs. This contradicts the fact that the column 
space of $B$ does not contain a nonzero constant vector. 
So $\rk B \ge 2$ indeed.

From $\rk B + \rk C = r$, we deduce that $\rk C \le 1$. If $C = 0$, then (1)
of theorem \ref{rk3} follows, so assume that $C \ne 0$. Then $\rk C = 1$ and
$\rk B = r - 1 = 2$. From lemmas \ref{Ccoldep2} and \ref{BCrkr}, we deduce
that the column space of $B$ contains a nonzero constant vector, which is a 
contradiction.

\end{itemize}
So it remains to prove the last claim. In case of (5) of theorem \ref{rk3}, 
the last claim follows from the proof of corollary \ref{rktrdeg}. Otherwise, 
the last claim of theorem \ref{rk3} follows in a similar manner as
the last claim of theorem \ref{rkr}.
\end{proof}

\begin{lemma} \label{F2lem}
Let $K$ be a field of characteristic $2$ and $L$ be an 
extension field of $K$. If theorem {\upshape\ref{rk3}} holds,
but for $L$ instead of $K$, then theorem {\upshape\ref{rk3}} 
holds.
\end{lemma}

\begin{proof}
Take $H$ as in theorem \ref{rk3}. Then $H$ is as in (1), (3) or (4)
of theorem \ref{rk3}, but with $L$ instead of $K$. 
We assume that $H$ is as in (3) of theorem \ref{rk3} with $L$ 
instead of $K$, because the other case follows in a similar manner
as lemma \ref{Fqlem}.

Notice that
$$
\big(\,0 ~ x_3 ~ x_2 ~ x_1 ~ y_4 ~ y_5 ~ \cdots ~ y_m\,\big) \cdot 
\jac \tilde{H} = 0
$$
Since $\jac \tilde{H} = S^{-1} (\jac H)|_{Tx} T$, it follows that 
$$
\big(\,0 ~ x_3 ~ x_2 ~ x_1 ~ y_4 ~ y_5 ~ \cdots ~ y_m\,\big) \cdot S \cdot \jac H = 0
$$
as well. 

Suppose first that $m = 4$. As $\rk \jac H = m-1$, there exists a nonzero
$v \in K(x)^m$ such that $\ker \big(\,v_1 ~ v_2 ~ v_3 ~ v_4\,\big)$ is equal 
to the column space of $\jac H$. Since the column space of $\jac H$ is contained
in $\ker \big((\,0~x_3~x_2~x_1\,) \cdot S\big)$, it follows that 
$\big(\, v_1 ~ v_2 ~ v_3 ~ v_4 \,\big)$ is dependent on $(\,0~x_3~x_2~x_1\,) \cdot S$. 
So
$$
v_1 (S^{-1})_{11} + v_2 (S^{-1})_{21} + v_3 (S^{-1})_{31} + v_4 (S^{-1})_{41} = 0
$$
and the components of $v$ are dependent over $L$. Consequently, the 
components of $v$ are dependent over $K$. So $\ker \big(\,v_1 ~ v_2 ~ v_3 ~ v_4\,)$
contains a nonzero vector over $K$, and so does the column space of $\jac H$. 
Now we can follow the same argumentation as in the first case in the proof
of theorem \ref{rk3}.

Suppose next that $m > 4$. Then the rows of $\jac H$ are dependent over
$L$. Hence the rows of $\jac H$ are dependent over $K$ as well. So we
may assume that the last row of $\jac H$ is zero. By induction on $m$,
$(H_1,H_2,\ldots,H_{m-1})$ is as $H$ in theorem \ref{rk3}. As $H_m = 0$,
we conclude that $H$ satisfies theorem \ref{rk3}.
\end{proof}

\section{rank 3 with nilpotency}

\begin{theorem} \label{rk3np}
Let $H \in K[x]^n$ be quadratic homogeneous, such that $\jac H$ is nilpotent and 
$\rk \jac H \le 3$. Then there exists a $T \in \GL_n(K)$, such that for 
$\tilde{H} := T^{-1} H(Tx)$, one of the following statements holds:
\begin{enumerate}[\upshape (1)]

\item $\jac \tilde{H}$ is lower triangular with zeroes on the diagonal;

\item $n \ge 5$, $\rk \jac H = 3$, and $\jac \tilde{H}$ is of the form
$$
\left( \begin{array}{ccc@{\qquad}ccc}
0 & x_5 & 0 & * & \cdots & * \\
x_4 & 0 & -x_5 & * & \cdots & * \\
0 & x_4 & 0 & * & \cdots & * \\[10pt]
0 & 0 & 0 & 0 & \cdots & 0 \\
\vdots & \vdots & \vdots & \vdots & & \vdots \\
0 & 0 & 0 & 0 & \cdots & 0
\end{array} \right)
$$

\item $n \ge 6$, $\rk \jac H = 3$, $\chr K = 2$ and $\jac \tilde{H}$ is 
of the form
$$
\left( \begin{array}{ccc@{\qquad}ccc@{\qquad}ccc}
0 & x_6 & 0 & 0 & 0 & x_2 & 0 & \cdots & 0 \\
x_5 & 0 & -x_6 & 0 & * & * & * & \cdots & * \\
0 & x_5 & 0 & 0 & x_2 & 0 & 0 & \cdots & 0 \\[10pt]
0 & 0 & 0 & 0 & x_6 & x_5 & 0 & \cdots & 0 \\
0 & 0 & 0 & 0 & 0 & 0 & 0 & \cdots & 0 \\
0 & 0 & 0 & 0 & 0 & 0 & 0 & \cdots & 0 \\[10pt]
0 & 0 & 0 & 0 & 0 & 0 & 0 & \cdots & 0 \\
\vdots & \vdots & \vdots & \vdots & \vdots & \vdots & \vdots & & \vdots \\
0 & 0 & 0 & 0 & 0 & 0 & 0 & \cdots & 0
\end{array} \right)
$$
 
\end{enumerate}
Furthermore, $x + H$ is tame if $\chr K \neq 2$, and 
there exists a tame invertible map 
$x + \bar{H}$ such that $\bar{H}$ is quadratic homogeneous and 
$\jac \bar{H} = \jac H$ in general.

Conversely, $\jac \tilde{H}$ is nilpotent in each of the three statements
above. Furthermore, $\rk \jac \tilde{H} = 3$ in {\upshape (2)},
$\rk \jac \tilde{H} = 4$ in {\upshape (3)} if $\chr K \neq 2$, and
$\rk \jac \tilde{H} = 3$ in {\upshape (3)} if $\chr K = 2$.
\end{theorem}

\begin{lemma} \label{lem3}
Let $H \in K[x]^3$ be homogeneous of degree $2$, such that
$\jac_{x_1,x_2,x_3} H$ is nilpotent. 
If $\jac_{x_1,x_2,x_3} H$ is not similar over $K$ to a triangular matrix, 
then $\jac_{x_1,x_2,x_3} H$ is similar over $K$ to a matrix of the form
$$
\left( \begin{array}{ccc}
0 & f & 0 \\ b & 0 & f \\ 0 & -b & 0        
\end{array} \right)
$$
where $f$ and $b$ are independent linear forms in $K[x_4,x_5,\ldots,x_n]$.
\end{lemma}

\begin{proof}
Suppose that $\jac_{x_1,x_2,x_3} H$ is not similar over $K$ to a triangular matrix.
Take $i$ such that the coefficient matrix of $x_i$ of $\jac_{x_1,x_2,x_3} H$ 
is nonzero, and define 
$$
N := \jac_{x_1,x_2,x_3} (H|_{x_i=x_i+1}) = (\jac_{x_1,x_2,x_3} H)|_{x_i=x_i+1}
$$
Then $N$ is nilpotent, and $N$ is not similar over $K$ to a triangular matrix. 
$N(0)$ is similar over $K$ to the matrix $N(0)$ in either (iii) or (iv) 
of \cite[Lemma 3.1]{1601.00579}, so it follows from
\cite[Lemma 3.1]{1601.00579} that $N$ is similar over $K$ to a matrix of the form 
$$
\left( \begin{array}{ccc}
0 & f+1 & 0 \\ b & 0 & f+1 \\ 0 & -b & 0        
\end{array} \right)
$$
where $b$ and $f$ are linear forms. $b$ and $f$ are independent, because the
coefficients of $x_i$ in $b$ and $f$ are $0$ and $1$ respectively. So
$\jac_{x_1,x_2,x_3} H$ is similar over $K$ to a matrix of the form
$$
\left( \begin{array}{ccc}
0 & f & 0 \\ b & 0 & f \\ 0 & -b & 0        
\end{array} \right)
$$
where $b$ and $f$ are independent linear forms.

Let $\bar{H}$ be the quadratic part with respect to $x_1, x_2, x_3$ of
$H$. Then $\jac_{x_1,x_2,x_3} \bar{H}$ is nilpotent.
By showing that $\bar{H} = 0$, we prove that $b$ and $f$ are contained in 
$K[x_4,x_5,\ldots,x_n]$. 

So assume that $\bar{H} \ne 0$. From \cite[Theorem 3.2]{1601.00579}, it follows
that we may assume that $\jac_{x_1,x_2,x_3} \bar{H}$ is lower triangular. If
the coefficient matrix of $x_2$ of $\jac_{x_1,x_2,x_3} \bar{H}$ is nonzero,
then it has rank $1$ because only the last row is nonzero, and so has the 
coefficient matrix of $x_2$ of 
$\jac_{x_1,x_2,x_3} H$. Otherwise, the coefficient matrix of $x_1$ of 
$\jac_{x_1,x_2,x_3} \bar{H}$ and $\jac_{x_1,x_2,x_3} H$ has rank $1$,
because only the first column is nonzero.

So we could have chosen $i$, such that the coefficient matrix of $x_i$ of 
$\jac_{x_1,x_2,x_3} H$ would have rank $1$. From \cite[Lemma 3.1]{1601.00579},
it follows that $\jac_{x_1,x_2,x_3} H$ is similar over $K$ to a triangular 
matrix. Contradiction, so $\bar{H} = 0$ indeed.
\end{proof}

\begin{lemma} \label{lem1}
Suppose that $H \in K[x]^n$, such that $\jac H$ is nilpotent.
\begin{enumerate}[\upshape (i)]

\item Suppose that $\jac H$ may only be nonzero in the first row and the 
first $2$ columns. Then there exists a $T \in \GL_n(K)$ such that for 
$\tilde{H} := T^{-1} H(Tx)$, the following holds.
\begin{enumerate}[\upshape (a)]

\item $\jac \tilde{H}$ may only be nonzero in the first row and the 
first $2$ columns (just like $\jac H$).

\item The Hessian matrix of the leading part with
respect to $x_2,x_3,\ldots,x_n$ of $\tilde{H}_1$ is the product of
a symmetric permutation matrix and a diagonal matrix. 

\item Every principal minor determinant of the leading principal minor 
matrix of size $2$ of $\jac \tilde{H}$ is zero.

\end{enumerate}

\item Suppose that $\chr K = 2$ and that $\jac H$ may only be nonzero 
in the first row and the first $3$ columns. Then there exists a 
$T \in \GL_n(K)$ such that for $\tilde{H} := T^{-1} H(Tx)$, the 
following holds.
\begin{enumerate}[\upshape (a)]

\item $\jac \tilde{H}$ may only be nonzero in the first row and the 
first $3$ columns (just like $\jac H$).

\item The Hessian matrix of the leading part with
respect to $x_2,x_3,\ldots,x_n$ of $\tilde{H}_1$ is the product of
a symmetric permutation matrix and a diagonal matrix. 

\item Every principal minor determinant of the leading principal minor 
matrix of size $3$ of $\jac \tilde{H}$ is zero.

\end{enumerate}
\end{enumerate}
\end{lemma}

\begin{proof}
From proposition \ref{evenrk}, it follows that there exists a
$T \in \Mat_{n,n}(K)$ of the form
$$
\left( \begin{array} {ccccc}
1 & 0 & 0 & \cdots & 0 \\
0 & 1 & 0 & \cdots & 0 \\
0 & * & 1 & \ddots & \vdots \\
\vdots & \vdots & \ddots & \ddots & 0 \\
0 & * & \cdots & * & 1
 \end{array} \right)
$$
such that the Hessian matrix of the leading part with
respect to $x_2,x_3,\ldots,x_n$ of $\tilde{H}_1 = H_1(Tx)$ is the 
product of a symmetric permutation matrix and a diagonal matrix. 

Furthermore, $\jac \tilde{H}$ may only be nonzero in the first row 
and the first $2$ or $3$ columns respectively (just like $\jac H$),
because of the form of $T$.
\begin{enumerate}[\upshape (i)]

\item 
Let $N$ be the principal minor matrix of size $2$ of $\jac \tilde{H}$.

Suppose first that $\jac \tilde{H}$ may only be nonzero in the first $2$ 
columns. From \cite[Lemma 1.2]{1601.00579}, it follows that $N$ is nilpotent.
On account of \cite[Theorem 3.2]{1601.00579}, $N$ is similar over $K$
to a triangular matrix. 

Hence the rows of $N$ are dependent over $K$.
If the second row of $N$ is zero, then (i) follows. If the second row of 
$N$ is not zero, then we may assume that the first row of $N$ is zero, and (i)
follows as well.

Suppose next that $\jac \tilde{H}$  may only be nonzero in the first row and 
the first $2$ columns, but not just the first $2$ columns. We distinguish two 
cases.
\begin{compactitem}

\item $\parder{}{x_2} \tilde{H}_1 \in K[x_1,x_2]$.

Since $\tr \jac \tilde{H} = 0$, we deduce that $\parder{}{x_1} \tilde{H}_1 \in 
K[x_1,x_2]$ as well. Let $G := \tilde{H}(x_1,x_2,0,0,\ldots,0)$. Then 
$\jac G = (\jac \tilde{H})|_{x_3=x_4=\cdots=x_n=0}$. Consequently,
the nonzero part of $\jac G$ is restricted to the first two columns.

So the leading principal minor matrix of size $2$ of $\jac G$ 
is nilpotent. But this minor matrix is just $N$, and just as for 
$\tilde{H}$ before, we may assume that only one row of $N$ is nonzero. 
This gives (i).

\item $\parder{}{x_2} \tilde{H}_1 \notin K[x_1,x_2]$.

Since $\hess (\tilde{H}_1|_{x_1=0})$ is the product of a permutation matrix and a diagonal matrix,
it follows that $\parder{}{x_2} \tilde{H}_1$ is a linear combination of $x_1$ and $x_i$,
where $i \ge 3$, such that $x_i$ does not occur in any other entry of $\jac \tilde{H}$.
Looking at the coefficient of $x_i^1$ in the sum of the determinants of the
principal minors of size $2$, we infer that $\parder{}{x_1} \tilde{H}_2 = 0$.

So the second row of $\jac \tilde{H}$ is dependent on $e_2\tp$. From a permuted version of 
\cite[Lemma 1.2]{1601.00579}, we infer that $\parder{}{x_2} \tilde{H}_2$ is nilpotent.
Hence the second row of $\jac \tilde{H}$ is zero. Since $\tr \jac \tilde{H} = 0$, we infer (i).

\end{compactitem}

\item 
Let $N$ be the principal minor matrix of size $3$ of $\jac \tilde{H}$.

Suppose first that $\jac \tilde{H}$ may only be nonzero in the first $3$ columns.
From \cite[Lemma 1.2]{1601.00579}, it follows that the leading principal
minor of size $3$ of $\jac \tilde{H}$ is nilpotent. On account of 
\cite[Theorem 3.2]{1601.00579}, $N$ is similar over $K$ to a triangular matrix. 
But for a triangular nilpotent Jacobian matrix of size $3$ over characteristic
$2$, the rank cannot be $2$. So $\rk N \le 1$.

Hence the rows of $N$ are dependent over $K$ in pairs.
If the second and the third row of $N$ are zero, then (ii) follows. 
If the second or the third row of $N$ is not zero, then we may assume that 
the first $2$ rows of $N$ are zero, and (ii) follows as well.

Suppose next that $\jac \tilde{H}$  may only be nonzero in the first row 
and the first $3$ columns, but not just the first $3$ columns. We distinguish 
three cases.
\begin{compactitem}

\item $\parder{}{x_2} \tilde{H}_1, \parder{}{x_3} \tilde{H}_1 \in K[x_1,x_2,x_3]$.

Using techniques of the proof of (i), we can reduce to the case where
$\jac \tilde{H}$ may only be nonzero in the first $3$ columns.

\item $\parder{}{x_2} \tilde{H}_1, \parder{}{x_3} \tilde{H}_1 \notin K[x_1,x_2,x_3]$.

Using techniques of the proof of (i), we can deduce that 
$\parder{}{x_1} \tilde{H}_2 = \parder{}{x_1} \tilde{H}_3 = 0$, 
and that $\jac_{x_2,x_3} (\tilde{H}_2,\tilde{H}_3)$ is nilpotent.

On account of \cite[Theorem 3.2]{1601.00579}, 
$\jac_{x_2,x_3} (\tilde{H}_2,\tilde{H}_3)$ is similar over $K$ 
to a triangular matrix. But for a triangular nilpotent Jacobian matrix 
of size $2$ over characteristic $2$, the rank cannot be $1$. 
So $\rk \jac_{x_2,x_3} (\tilde{H}_2,\tilde{H}_3) = 0$. 
Consequently, the last two rows of $N$ are zero, and (ii) follows.

\item None of the above.

Assume without loss of generality that $\parder{}{x_2} \tilde{H}_1 \in 
K[x_1,x_2,x_3]$ and $\parder{}{x_3} \tilde{H}_1 \notin K[x_1,x_2,x_3]$. 
Since $\hess (\tilde{H}_1|_{x_1=0})$ is the product of a permutation matrix and a diagonal 
matrix, it follows that $\parder{}{x_3} \tilde{H}_1$ is a linear combination of 
$x_1$ and $x_i$, where $i \ge 4$, such that $x_i$ does not occur in any other 
entry of $\jac \tilde{H}$.

Looking at the coefficient of $x_i^1$ in the sum of the determinants of the
principal minors of size $2$, we infer that $\parder{}{x_1} \tilde{H}_3 = 0$. If
$\parder{}{x_1} \tilde{H}_2 = 0$ as well, then we can advance as above, so assume that
$\parder{}{x_1} \tilde{H}_2 \neq 0$. 

Looking at the coefficient of $x_i^1$ in the 
sum of the determinants of the principal minors of size $3$, we infer that 
$\parder{}{x_2} \tilde{H}_3 = 0$. So the third row of $\jac \tilde{H}$ is dependent on 
$e_3\tp$. From a permuted version of \cite[Lemma 1.2]{1601.00579}, we infer that 
$\parder{}{x_3} \tilde{H}_3$ is nilpotent. Hence the third row of 
$\jac \tilde{H}$ is zero. 

From $\tr \jac \tilde{H} = 0$, we deduce that
$\parder{}{x_1} \tilde{H}_1 = -\parder{}{x_2} \tilde{H}_2$. We show that 
\begin{equation} \label{DN0}
\parder{}{x_1} \tilde{H}_1 = \parder{}{x_2} \tilde{H}_2 = 0
\end{equation}
For that purpose, suppose that $\parder{}{x_1} \tilde{H}_1 \neq 0$.
Since  $\parder{}{x_1} x_1^2 = 0$, the coefficient of $x_1$ in 
$\parder{}{x_1} \tilde{H}_1$ is zero. Similarly, the coefficient of $x_2$ in 
$\parder{}{x_2} \tilde{H}_2$ is zero. As $\tilde{H}_2 \in K[x_1,x_2,x_3]$, we
infer that
$$
\parder{}{x_1} \tilde{H}_1 = -\parder{}{x_2} \tilde{H}_2 \in K x_3 \setminus \{0\}
$$
Looking at the 
coefficient of $x_3^2$ in the sum of the determinants of the 
principal minors of size $2$, we deduce that the coefficient of $x_3^2$ in
$(\parder{}{x_i} \tilde{H}_1) \cdot (\parder{}{x_1} \tilde{H}_i)$ is nonzero.

Consequently, the coefficient of $x_3^3$ in 
$(\parder{}{x_i} \tilde{H}_1) \cdot (\parder{}{x_2} \tilde{H}_2) \cdot 
(\parder{}{x_1} \tilde{H}_i) \in K x_3^3 \setminus \{0\}$ is nonzero. 
This contributes to the coefficient of $x_3^3$ in the sum of the determinants 
of the principal minors of size $3$. Contradiction, because this contribution 
cannot be canceled.

So \eqref{DN0} is satisfied. We show that in addition,
\begin{equation} \label{N120}
\parder{}{x_2} \tilde{H}_1 = 0
\end{equation}
The coefficient of $x_1$ of $\parder{}{x_2} \tilde{H}_1$ is zero, because of \eqref{DN0}.
The coefficient of $x_2$ of $\parder{}{x_2} \tilde{H}_1$ is zero, because 
$\parder{}{x_2} x_2^2 = 0$.
The coefficient of $x_3$ of $\parder{}{x_2} \tilde{H}_1$ is zero, because
the coefficient of $x_2$ of $\parder{}{x_3} \tilde{H}_1 \in K x_1 + K x_i$ is zero.
So \eqref{N120} is satisfied as well.

Recall that the third row of $N$ is zero. 
From \eqref{DN0} and \eqref{N120}, it follows that the diagonal and the second column
of $N$ are zero as well. Hence every principal minor determinant of $N$ is zero, which gives (ii).
\qedhere

\end{compactitem}
\end{enumerate}
\end{proof}

\begin{lemma} \label{lem2}
Let $\tilde{H}$ be as in lemma \ref{lem1}. Suppose that $\jac \tilde{H}$
has a principal minor matrix $M$ of which the determinant is nonzero. Then 
\begin{enumerate}[\upshape (a)]

\item $\tilde{H}$ is as in {\upshape (ii)} of lemma \ref{lem1};

\item rows $2$ and $3$ of $\jac \tilde{H}$ are zero;

\item $M$ has size $2$ and $x_2 x_3 \mid \det M$;

\item Besides $M$, there exists exactly one principal minor matrix $M'$ 
of size $2$ of $\jac H$, such that $\det M' = - \det M$.

\end{enumerate}
\end{lemma}

\begin{proof}
Take $N$ as in (i) or (ii) of lemma \ref{lem1} respectively. 
Then $M$ is not a principal minor matrix of $N$. So if $M$ does not 
contain the upper left corner of $\jac \tilde{H}$, then the last column
of $M$ is zero. Hence $M$ does contain the upper left corner of 
$\jac \tilde{H}$. 

If $M$ has two columns outside the column range of $N$,
then both columns are dependent on $e_1$. So $M$ has exactly one column
outside the column range of $N$, say column $i$. 
\begin{enumerate}[(i)]

\item Suppose first that $\tilde{H}$ is as in (i) of lemma \ref{lem1}.
Then either $M$ has size $2$ with row and column indices $1$ and $i$,
or $M$ has size $3$ with row and column indices $1$, $2$ and $i$.

The coefficient of $x_1$ in the upper right corner of $M$ is zero, because
$N_{11} = 0$. Hence the upper right corner 
of $M$ is of the form $c x_j$ for some nonzero $c \in K$ and a $j \ge 2$.
If $j \ge 3$, then $\det M$ is equal to the sum of the terms which are 
divisible by $x_j$ in the sum of the determinants of the principal minors of 
the same size as $M$, which is zero. 

So $j = 2$. Now $M$ is the only principal minor matrix of its size, of which
the determinant is nonzero, because if we would take another $i$, then $j$ 
changes as well. Contradiction, because the sum of the determinants of the 
principal minors of the same size as $M$ is zero. 

\item Suppose next that $\tilde{H}$ is as in (ii) of lemma \ref{lem1}.
If the second row of $\jac\tilde{H}$ is nonzero, then the coefficient of
$x_1 x_3$ in $\tilde{H}_2$ is nonzero, because $N_{22} = 0$. 
If the third row of $\jac\tilde{H}$ 
is nonzero, then the coefficient of $x_1 x_2$ in $\tilde{H}_3$ is nonzero,
because $N_{33} = 0$.
Since every principal minor determinant of $N$ is zero, we infer that 
$N_{23} N_{32} = 0$, so either the second or the third row of $\jac\tilde{H}$ is zero.

Assume without loss of generality that the second row of $\jac \tilde{H}$
is zero. Then either $M$ has size $2$ with row and column indices $1$ and $i$,
or $M$ has size $3$ with row and column indices $1$, $3$ and $i$.
The upper right corner of $M$ is of the form $c x_j$ for some nonzero 
$c \in K$, and with the techniques in (i) above, we see that $2 \le j \le 3$.

Furthermore, we infer with the techniques in (i) above that $\jac \tilde{H}$ 
has another principal minor matrix $M'$ of the same size as $M'$, of which the
determinant is nonzero as well. The upper right corner of $M'$ can only be
of the form $c' x_{5-j}$ for some nonzero $c' \in K$.

It follows that $N_{12} \ne 0$ and $N_{13} \ne 0$. Consequently, $N_{21} = N_{31} = 0$.
This is only possible if both the second and the third
row of $\jac \tilde{H}$ are zero. So $M$ has size $2$, and claims (c) and (d)
follow. \qedhere

\end{enumerate}
\end{proof}

\begin{lemma} \label{lem4}
Let $n = 4$, 
$$
\tilde{H} = \left( \begin{array}{c}
x_1 x_3 + c x_2 x_4 \\ x_2 x_3 - x_1 x_4 \\
\frac12 x_3^2 + \frac{c}2 x_4^2 \\ \frac12 x_1^2 + \frac{c}2 x_2^2 
\end{array} \right)
\qquad \mbox{and} \qquad
\jac \tilde{H} = \left( \begin{array}{cccc}
x_3 & c x_4 & x_1 & c x_2 \\
- x_4 & x_3 & x_2 & -x_1 \\
0 & 0 & x_3 & c x_4 \\
x_1 & c x_2 & 0 & 0
\end{array} \right)
$$
as in lemma \ref{rk3calc} (with $c \neq 0$). 

Let $M \in \Mat_{4,4}(K)$, such that $\deg \det \big(\jac{\tilde{H}} + M\big) \le 2$. 
Then there exists a translation $G$, such that 
$$
\tilde{H}\big(G(x)\big) - \big(\tilde{H} + Mx\big) \in K^4
$$
In particular, $\det \big(\jac{\tilde{H}} + M\big) = \det \jac \big(\tilde{H} + Mx\big) = 0$.
\end{lemma}

\begin{proof}
Since the quartic part of $\det (\jac{\tilde{H}} + M)$ is zero, we deduce that
$\det (\jac{\tilde{H}}) = 0$.
By way of completing the squares, we can choose a translation $G$ such that 
the linear part of $F := \tilde{H}(G^{-1}(x)) + M\,G^{-1}(x)$ is of the form
$$
\left( \begin{array}{cccccccc}
 a_1 x_1 &+& b_1 x_2 &+& c_1 x_3 &+& d_1 x_4 \\
 a_2 x_1 &+& b_2 x_2 &+& c_2 x_3 &+& d_2 x_4 \\
 a_3 x_1 &+& b_3 x_2 & &         & &         \\
         & &         & & c_4 x_3 &+& d_4 x_4
\end{array}\right)
$$ 
Notice that $\deg \det \jac F \le 2$. Looking at the coefficients of
$x_1^3$, $x_2^3$, $x_3^3$, and $x_4^3$ of $\det \jac F$, we see that 
$b_3 = a_3 = d_4 = c_4 = 0$. Looking at the coefficients of
$x_1^2 x_3$, $x_1 x_3^2$, $x_2^2 x_4$, and $x_2 x_4^2$ of $\det \jac F$, 
we see that $b_1 = d_1 = a_1 = c_1 = 0$. Looking at the coefficients of
$x_1^2 x_4$, $x_1 x_4^2$, $x_2^2 x_3$, and $x_2 x_3^2$ of $\det \jac F$, 
we see that $b_2 = c_2 = a_2 = d_2 = 0$. 

So $F$ has trivial linear part, and $\tilde{H} - F \in K^4$. 
Hence $\tilde{H}(G) - F(G) \in K^4$, as claimed. The last claim follows from
$\det (\jac{\tilde{H}}) = 0$.
\end{proof}

\begin{proof}[Proof of theorem \ref{rk3np}]
From \cite[Theorem 3.2]{1601.00579}, it follows that (1) is satisfied if
$\rk \jac H \le 2$. So assume that $\rk \jac H = 3$. Then we can follow the cases
of theorem \ref{rk3}.
\begin{itemize}

\item $H$ is as in (1) of theorem \ref{rk3}.

Let $\tilde{H} = S H(S^{-1}x)$. Then only the first $3$ rows of 
$\jac \tilde{H}$ may be nonzero. If the leading principal minor 
matrix $N$ of size $3$ of $\jac \tilde{H}$ is similar over $K$ 
to a triangular matrix, then so is $\jac \tilde{H}$ itself. 
So assume that $N$ is not similar over $K$ to a triangular matrix.
From lemma \ref{lem3}, it follows that $N$ is similar over $K$ to
a matrix of the form
\begin{equation} \label{eq3}
\left( \begin{array}{ccc}
0 & f & 0 \\ b & 0 & f \\ 0 & -b & 0        
\end{array} \right)
\end{equation}
where $f$ and $b$ are independent linear forms in $K[x_4,x_5,\ldots,x_n]$.

Consequently, we can choose $S$, such that for 
$\tilde{H} := S H(S^{-1}x)$, only the first $3$ rows of $\jac \tilde{H}$ 
are nonzero, and the leading principal minor matrix of size $3$ is as
in \eqref{eq3}. If we negate the third row and the third column of \eqref{eq3}, 
and replace $b$ by $x_4$ and $f$ by $x_5$, then we get 
$$
\left( \begin{array}{ccc}
0 & x_5 & 0 \\ x_4 & 0 & -x_5 \\ 0 & x_4 & 0        
\end{array} \right)
$$
We can even choose $S$ such that the leading principal minor matrix of size $3$
of $\jac \tilde{H}$ is as above. So (2) of theorem \ref{rk3np} is satisfied for
$T = S^{-1}$.

If we replace $\tilde{H}_2$ by $0$, then all principal minor determinants of 
$\jac \tilde{H}$ become zero. On account of \cite[Lemma 1.2]{MR1210415}, 
$\jac \tilde{H}$ becomes permutation similar to a triangular matrix. 
From proposition \ref{tameZ} below, we infer that $x + H$ is tame
if $\chr K = 0$.

\item $H$ is as in (2) of theorem \ref{rk3} or as in (2) of theorem \ref{rkr}.

Let $\tilde{H} = T^{-1} H(Tx)$. Then the rows of $\jac_{x_3,x_4,\ldots,x_n} \tilde{H}$ 
are dependent over $K$ in pairs. Suppose first that the first $2$ rows of 
$\jac_{x_3,x_4,\ldots,x_n} \tilde{H}$ are zero. Then we can choose $T$ such that
only the last row of $\jac_{x_3,x_4,\ldots,x_n} \tilde{H}$ may be nonzero.

From \cite[Lemma 1.2]{1601.00579}, it follows that 
leading principal minor matrix $N$ of size $2$ of $\jac \tilde{H}$ is 
nilpotent as well. On account of \cite[Theorem 3.2]{1601.00579}, $N$ is
similar over $K$ to a triangular matrix. From \cite[Corollary 1.4]{1601.00579}, 
we deduce that $\jac \tilde{H}$ is similar over $K$ to a triangular matrix.
So we can choose $T$ such that $\jac \tilde{H}$ is lower triangular, and
(1) is satisfied. Furthermore, $x + H$ is tame if $\chr K = 0$.

Suppose next that the first $2$ rows of 
$\jac_{x_3,x_4,\ldots,x_n} \tilde{H}$ are not both zero. Then we can choose 
$T$ such that only the first row of $\jac_{x_3,x_4,\ldots,x_n} \tilde{H}$ may 
be nonzero. From lemma \ref{lem1} (i) and lemma \ref{lem2}, we infer that we 
can choose $T$ such that every principal minor of $\jac \tilde{H}$ has determinant 
zero. From \cite[Lemma 1.2]{MR1210415}, it follows that $\jac \tilde{H}$ is 
permutation similar to a triangular matrix. So (1) is satisfied.
Furthermore, $x + H$ is tame if $\chr K = 0$.

\item $H$ is as in (3) of theorem \ref{rk3} or as in (3) of theorem \ref{rkr}.

Let $\tilde{H} = T^{-1} H(Tx)$. Then the rows of $\jac_{x_4,x_5,\ldots,x_n} \tilde{H}$ 
are dependent over $K$ in pairs. Suppose first that the first $3$ rows of 
$\jac_{x_4,x_5,\ldots,x_n} \tilde{H}$ are zero. Then we can choose $T$ such that
only the last row of $\jac_{x_4,x_5,\ldots,x_n} \tilde{H}$ may be nonzero,
and just as above, (1) is satisfied. Furthermore, $x + H$ is tame if $\chr K = 0$.

Suppose next that the first $3$ rows of 
$\jac_{x_4,x_5,\ldots,x_n} \tilde{H}$ are not all zero. Then we can choose 
$T$ such that only the first row of $\jac_{x_4,x_5,\ldots,x_n} \tilde{H}$ may 
be nonzero. If we can choose $T$ such that every principal minor of 
$\jac \tilde{H}$ has determinant zero, then (1) is satified, just as before.
Furthermore, $x + H$ is tame if $\chr K = 0$.

So assume that we cannot choose $T$ such that every principal minor of 
$\jac \tilde{H}$ has determinant zero. From lemma \ref{lem1} (ii) and 
lemma \ref{lem2}, we infer that we can choose $T$ such that $\tilde{H}$
is as in lemma \ref{lem2}. More precisely, we can choose $T$ such that
$\jac \tilde{H}$ is of the form
\begin{equation} \label{eq2}
\left( \begin{array}{cccccccc}
0 & x_4 & -x_5 & x_2 & -x_3 & * & * & \cdots \\
0 & 0 & 0 & 0 & 0 & 0 & 0 & \cdots \\
0 & 0 & 0 & 0 & 0 & 0 & 0 & \cdots \\
x_3 & * & * & 0 & 0 & 0 & 0 & \cdots \\
x_2 & * & * & 0 & 0 & 0 & 0 & \cdots \\
* & * & * & 0 & 0 & 0 & 0 & \cdots \\
* & * & * & 0 & 0 & 0 & 0 & \cdots \\
\vdots & \vdots & \vdots & \vdots & 
\vdots & \vdots & \vdots & \ddots 
\end{array} \right)
\end{equation}
Since $\parder{}{x_1} x_1^2 = 0$, the coefficients of $x_1$
of the starred entries in the first column of \eqref{eq2} are zero.
Hence we can clean these starred entries by way of row operations
with rows $4$ and $5$. 

We can also clean them by way of a linear
conjugation, because if a starred entry in the first column is nonzero,
the transposed entry in the first row is zero, so the corresponding
column operations which are induced by the linear conjugation will not have
any effect.

If rows $1$, $4$, and $5$ remain as the only nonzero rows, then $H$ is as in (1) 
of theorem \ref{rk3}, which is the first case. So assume that another row remains
nonzero. By way of additional row operations and associated column operations, 
we can get $\jac \tilde{H}$ of the form
\begin{equation} \label{eq2a}
\left( \begin{array}{cccccccc}
0 & x_4 & -x_5 & x_2 & -x_3 & * & * & \cdots \\
0 & 0 & 0 & 0 & 0 & 0 & 0 & \cdots \\
0 & 0 & 0 & 0 & 0 & 0 & 0 & \cdots \\
x_3 & 0 & x_1 & 0 & 0 & 0 & 0 & \cdots \\
x_2 & x_1 & 0 & 0 & 0 & 0 & 0 & \cdots \\
0 & x_3 & x_2 & 0 & 0 & 0 & 0 & \cdots \\
0 & 0 & 0 & 0 & 0 & 0 & 0 & \cdots \\
0 & 0 & 0 & 0 & 0 & 0 & 0 & \cdots \\
\vdots & \vdots & \vdots & \vdots & 
\vdots & \vdots & \vdots & \ddots 
\end{array} \right)
\end{equation}
where we do not maintain (b) of lemma \ref{lem1} (ii) any more.
Hence $P^{-1} \tilde{H}(Px)$ is as $\tilde{H}$ in (3) for a suitable 
permutation matrix $P$.

In a similar manner as in the case where $H$ is as in (1) of theorem 
\ref{rk3}, we infer that $x + H$ is tame if $\chr K = 0$.

\item $H$ is as in (5) of theorem \ref{rk3}.

Then only the first $4$ columns of $\tilde{H} = T^{-1} H(Tx)$ may be nonzero.
Hence the leading principal minor matrix $N$ of size $4$ of $\jac \tilde{H}$
is nilpotent. We distinguish two subcases.
\begin{compactitem}

\item The rows of $N$ are linearly independent over $K$.

Then there exists an $U \in \GL_4(K)$, such that $U N$ is as $\jac \tilde{H}$
in lemma \ref{lem4}. Furthermore,
$$
\det (UN + U) = \det U \det (N + I_4) = \det U \in K^{*}
$$
So $\det (UN + U) \ne 0$ and $\deg \det (UN + U) \le 2$. 
This contradicts lemma \ref{lem4}.

\item The rows of $N$ are linearly independent over $K$.

Then we can apply the first case above on the map 
$(\tilde{H}_1,\tilde{H}_2,\tilde{H}_3,\tilde{H}_4)$.
Since $N$ has less than $5$ columns, the case where $N$ is not similar 
over $K$ to a triangular matrix cannot occur in the case above. So
$N$ is similar over $K$ to a triangular matrix, and so are $\jac \tilde{H}$
and $\jac H$. Furthermore, $x + H$ is tame if $\chr K = 0$.

\end{compactitem}
\end{itemize}
So it remains to prove the last claim in the case $\chr K \neq 0$.
This follows by way of techniques in the proof of theorem \ref{dim5} in the 
next section.
\end{proof}

\begin{proposition} \label{tameZ}
Let $R$ be an integral domain of characteristic zero, and let $F \in R[x]^n$ 
be a polynomial map which is invertible over $R$. Let $\tilde{F} \in R[x]^n$, 
such that only the $i$\textsuperscript{th} component of $\tilde{F}$ is 
different from that of $F$, and such that $\det \jac \tilde{F} = \det \jac F$. 

Then $\tilde{F}(F^{-1})$ is an elementary invertible polynomial map over $R$. 
In particular, $\tilde{F}$ is invertible over $R$.
\end{proposition}

\begin{proof}
Assume without loss of generality that $i = n$. Let $K$ be the fraction field of $R$. 
Then $\rk \jac H = \trdeg_K K(H)$, because $K$ has characteristic zero.
Since 
$$
\rk \jac (F_1,F_2,\ldots,F_{n-1},0) = n-1 = \rk \jac (F_1,F_2,\ldots,F_{n-1},\tilde{F}_n-F_n)
$$
it follows that $\tilde{F}_n-F_n$ is algebraically dependent
over $K$ on $F_1,F_2,\ldots,F_{n-1}$. As $F$ is invertible over $R$, 
$$
\tilde{F}_n-F_n \in R[F_1,F_2,\ldots,F_{n-1}]
$$
So $\tilde{F}_n(F^{-1})-x_n \in R[x_1,x_2,\ldots,x_{n-1}]$. Furthermore, 
$\tilde{F}_i(F^{-1}) = F_i(F^{-1}) = x_i$ for all $i < n$ by assumption, so 
$\tilde{F}(F^{-1})$ is elementary.
\end{proof}

\section{rank 4 with nilpotency}

Let $M \in \Mat_n(K)$ be nilpotent and $v \in K^n$ be nonzero.
Define the {\em image exponent} of $v$ with respect to $M$ as
$$
\ie (M,v) = \ie_K (M,v) := \max \{i \in \N \mid M^i v \ne 0\}
$$ 
and the {\em preimage exponent} of $v$ with respect to $M$ as
$$
\pe (M,v) = \pe_K (M,v) := \max \{i \in \N \mid M^i w = v \mbox{ for some } w \in K^n\}
$$

\begin{theorem} \label{dim5}
Let $\chr K \neq 2$ and $n = 5$. Suppose that $H \in K[x]^5$, 
such that $\jac H$ is nilpotent and $\rk \jac H \ge 4$. Then
$\rk \jac H = 4$, and
there exists a $T \in GL_5(K)$, such that $\jac \big(T^{-1}H(Tx)\big)$ 
is either triangular with zeroes on the diagonal, or of one of the 
following forms.
$$
\left( \begin{array}{ccccc} 
0 & 0 & 0 & 0 & 0 \\
* & 0 & 0 & 0 & 0 \\
0 & x_4 & 0 & x_2 & 0 \\
x_3 & -x_5 & x_1 & 0 & -x_2 \\
* & * & 0 & x_1 & 0
\end{array} \right)
\quad
\left( \begin{array}{ccccc} 
0 & 0 & 0 & 0 & 0 \\
* & 0 & 0 & x_4 & 0 \\
x_2 & x_1 & 0 & -x_5 & -x_4 \\
x_3 & 0 & x_1 & 0 & x_5 \\
* & 0 & 0 & x_1 & 0
\end{array} \right)
$$
Furthermore, $x + H$ is tame.
\end{theorem}

\begin{proof}
Suppose first that $K$ is infinite.
From \cite[Theorem 4.4]{1609.09753}, it follows that $\pe_{K(x)}(\jac H,x) = 0$.
As $\rk \jac H = n - 1$, $\ie_{K(x)}(\jac H,x) + \pe_{K(x)}(\jac H,x) = n-1$, so
$\ie_{K(x)}(\jac H,x) = n - 1$. From \cite[Corollary 4.3]{1609.09753}, it follows
that there exists a $T \in \GL_5(K)$, such that for 
$\tilde{H} := T^{-1}H(Tx)$, we have
$$
(\jac \tilde{H})|_{x = e_1} = 
\left( \begin{array}{ccccc} 
0 & 0 & 0 & 0 & 0 \\
1 & 0 & 0 & 0 & 0 \\
0 & 1 & 0 & 0 & 0 \\
0 & 0 & 1 & 0 & 0 \\
0 & 0 & 0 & 1 & 0
\end{array} \right)
$$
Using this, one can compute all solutions, and match them against 
the given classification. We did this with Maple 8, see {\tt dim5qdr.pdf}.

To prove that $x + H$ is tame, we show that the maps of the classification 
are tame. If $\jac H$ is similar over $K$ to a triangular matrix, 
then $x + H$ is tame. So assume that $\jac H$ is not similar over $K$ to
a triangular matrix. If 
$$
G = \left( \begin{array}{c}
0 \\ z_1 x_1^2 \\ x_2 x_4 \\ x_1 x_3 - x_2 x_5 \\
z_2 x_1^2 + z_3 x_1 x_2 + z_4 x_2^2 + x_1 x_4
\end{array} \right) 
\quad \mbox{or} \quad
G = \left( \begin{array}{c}
0 \\ z_1 x_1^2 + \tfrac12 x_4^2 \\ x_1 x_2 - x_4 x_5 \\ 
x_1 x_3 - \tfrac12 x_5^2 \\
z_2 x_1^2 + x_1 x_4
\end{array} \right) 
$$
then we can apply proposition \ref{tameZ} with $R = \Z[z_1,z_2,z_3,z_4]$
and $i = 4$, to obtain that $x + G$ is the composition of an elementary map 
and a map $x + \tilde{G}$ for which $\jac \tilde{G}$ is permutation similar to
a triangular matrix with zeroes on the diagonal. So $x + G$ is tame over $R$.
Hence $x + G$ modulo $I$ is tame over $R/I$ for any ideal $I$ of $R$.
Since $x + H$ has this form up to conjugation with a linear map, we infer
that $x + H$ is tame.

Suppose next that $K$ is finite, and let $L$ be an infinite extension
field of $K$. If $\jac H$ is similar over $L$ to a triangular matrix,
then by \cite[Proposition 1.3]{1601.00579}, $\jac H$ is similar over $K$ 
to a triangular matrix as well. So assume that $\jac H$ is not similar over 
$K$ to a triangular matrix. Then one can check for the solutions over $L$
that \cite[Theorem 5.2]{1609.09753} is satisfied, which we did. So 
\cite[Corollary 4.3]{1609.09753} holds over $K$ as well, and so does
this theorem.
\end{proof}

\begin{theorem} \label{dim6}
Let $\chr K = 2$ and $n = 6$. Suppose that $H \in K[x]^6$, 
such that $\jac H$ is nilpotent and $\rk \jac H \ge 4$. Then
$\rk \jac H = 4$, and 
there exists a $T \in GL_6(K)$, such that $\jac \big(T^{-1}H(Tx)\big)$ 
is either triangular with zeroes on the diagonal, or of one of the 
following forms.
\begin{gather*}
\left( \begin{array}{cccccc} 
0 & 0 & 0 & 0 & 0 & 0 \\
0 & 0 & 0 & 0 & 0 & 0 \\
x_2 & x_1 & 0 & 0 & 0 & 0 \\
0 & 0 & x_5 & 0 & x_3 & 0 \\
* & * & * & x_2 & 0 & -x_3 \\
* & * & * & 0 & x_2 & 0
\end{array} \right) 
\quad
\left( \begin{array}{cccccc} 
0 & 0 & 0 & 0 & 0 & 0 \\
0 & 0 & 0 & 0 & 0 & 0 \\
0 & x_5 & 0 & 0 & x_2 & 0 \\
x_3 & \!\!-c x_5 - x_6\!\! & x_1 & 0 & -c x_2 & -x_2 \\
x_4 & c x_6 & 0 & x_1 & 0 & c x_2 \\
x_5 & 0 & 0 & 0 & x_1 & 0
\end{array} \right) \\
\left( \begin{array}{cccccc} 
0 & 0 & 0 & 0 & 0 & 0 \\
0 & 0 & 0 & 0 & 0 & 0 \\
0 & x_4 & 0 & x_2 & 0 & 0 \\
x_3 & -x_5 & x_1 & 0 & -x_2 & 0 \\
x_4 & 0 & 0 & x_1 & 0 & 0 \\
* & * & * & * & * & 0
\end{array} \right)
\quad
\left( \begin{array}{cccccc} 
0 & 0 & 0 & 0 & 0 & 0 \\
0 & 0 & 0 & x_5 & x_4 & 0 \\
x_2 & x_1 & 0 & -x_6 & -2x_5 & -x_4 \\
x_3 & 0 & x_1 & 0 & x_6 & x_5 \\
x_4 & 0 & 0 & x_1 & 0 & 0 \\
x_5 & 0 & 0 & 0 & x_1 & 0
\end{array} \right)
\end{gather*}
where $c \in \{0,1\}$. Furthermore, there exists a tame invertible map 
$x + \bar{H}$ such that $\bar{H}$ is quadratic homogeneous and 
$\jac \bar{H} = \jac H$.
\end{theorem}

\begin{proof}
Suppose first that $K$ is infinite.
From \cite[Theorem 4.4]{1609.09753}, it follows that $\pe(\jac H,x) = 0$.
Since $\ie_{K(x)}(\jac H,x) = 0$ as well, it follows from 
\cite[Theorem 4.4]{1609.09753} that there exists a $T \in \GL_5(K)$, 
such that for $\tilde{H} := T^{-1}H(Tx)$, we have
$$
(\jac \tilde{H})|_{x = e_1} = 
\left( \begin{array}{cccccc} 
0 & 0 & 0 & 0 & 0 & 0 \\
0 & 0 & 0 & 0 & 0 & 0 \\
0 & 1 & 0 & 0 & 0 & 0 \\
0 & 0 & 1 & 0 & 0 & 0 \\
0 & 0 & 0 & 1 & 0 & 0 \\
0 & 0 & 0 & 0 & 1 & 0
\end{array} \right)
$$
Using this, one can compute all solutions, and match them against 
the given classification. We did this with Maple 8, see {\tt dim6chr2qdr.pdf}.

To prove that $x + \bar{H}$ is tame for some quadratic homogeneous $\bar{H}$
such that $\jac \bar{H} = \jac H$, we can use the same arguments as in
the proof of theorem \ref{dim5}. Namely, we apply proposition \ref{tameZ} 
with $R = \Z[z_1,z_2,z_3,z_4,z_5,z_6,z_7,z_8,z_9,z_{10}]$ and
$i=5$, $i=5$, $i=4$, and $i=4$ respectively.

Suppose next that $K$ is finite, and let $L$ be any infinite extension
field of $K$. If $\jac H$ is similar over $L$ to a triangular matrix,
then by \cite[Proposition 1.3]{1601.00579}, $\jac H$ is similar over $K$ 
to a triangular matrix as well. So assume that $\jac H$ is not similar over 
$K$ to a triangular matrix. We distinguish two cases.
\begin{itemize}
 
\item The columns of $\jac H$ are dependent over $L$.

Then the columns of $\jac H$ are dependent over $K$ as well. So we may
assume that the last column of $\jac H$ is zero. Then the leading principal
minor $N$ of size $5$ of $\jac H$ is nilpotent as well. Just like
$\rk \jac H = 6 - 2 = 4$, we deduce that $\rk N = 5 - 2 = 3$. So we
can apply theorem \ref{rk3np} on $N$, to infer that we can get 
$\jac \tilde{H}$ of the form of the third matrix.

\item The columns of $\jac H$ are independent over $L$.
 
Then one can check for the solutions over $L$
that lemma \ref{lem5} below is satisfied, which we did. So 
there exists a $v \in K^6$, such that $(\jac H)|_{x=v}^4 \neq 0$.
Hence the Jordan Normal Form of $(\jac H)|_{x=v}$ has Jordan blocks of
size $1$ and $5$, just like that of $\jac H$. 

Furthermore, 
$\ie\big((\jac H)|_{x=v},v\big) = 0 = \ie_{K(x)}(\jac H,x)$, and it follows 
from \cite[Theorem 4.4]{1609.09753} that $\pe\big((\jac H)|_{x=v},v\big) = 0
= \pe_{K(x)}(\jac H,x)$. Consequently, one can verify that 
\cite[Theorem 4.2]{1609.09753} holds over $K$ as well, and so does 
this theorem. \qedhere

\end{itemize}
\end{proof}

\begin{lemma} \label{lem5}
Let $L$ be an extension field of $K$. Suppose that $H \in K[x]^n$
and $T \in \GL_n(L)$, such that for $\tilde{H} := T^{-1}H(Tx)$,
the ideal generated by the entries of $(\jac \tilde{H})^s$ 
contains a power of a polynomial $f$.

Then there exists a $v \in K^n$, such that $(\jac H)|_{x=v}^s \neq 0$,
in the following cases.
\begin{enumerate}[\upshape (i)]

\item $\#K \ge \deg f + 1$;

\item $f$ is homogeneous and $\#K \ge \deg f$.
 
\end{enumerate}
\end{lemma}

\begin{proof}
From  \cite[Lemma 5.1]{1310.7843}, it follows that there exists a $v \in K^n$
such that $f(T^{-1}v) \ne 0$. Let $I$ be the ideal over $L$ generated by the 
entries of $(\jac \tilde{H})^s$. As the radical of $I$ contains $f$, the 
radical of $I(T^{-1}v)$ contains $f(T^{-1}v)$. 
$I(T^{-1}v)$ is generated over $L$ by the entries of 
$(\jac \tilde{H})|_{x=T^{-1}v}^s$ and
$$
T (\jac \tilde{H})|_{x=T^{-1}v}^s T^{-1} = (\jac H)|_{x=v}^s
$$
From $f(T^{-1}v) \ne 0$, it follows that $I(T^{-1}v) \ne 0$ and 
$(\jac H)|_{x=v}^s \ne 0$.
\end{proof}

The cardinality of $K$ may be too small for the computational method
of \cite[Theorem 4.2]{1609.09753}: the following maps $H$ do not satisfy 
\cite[Theorem 4.2]{1609.09753} and neither lemma \ref{Irlem}:
\begin{itemize}

\item $\#K = 3$ and $H = (0,\frac12 x_1^2,\frac12 x_2^2,(x_1+x_2)x_3,(x_1-x_2)x_4)$;

\item $\#K = 2$ and $H = (0,0,x_1 x_2,x_1 x_3,x_2 x_4,(x_1-x_2)x_5)$;

\item $\#K = 2$ and $H = (0,0,x_1 x_4,x_1 x_3-x_2 x_5,x_2 x_4,(x_1-x_4)x_5)$.

\end{itemize}

\begin{theorem}
Let $H$ be quadratic homogeneous such that $\jac H$ is nilpotent and $\rk \jac H \le 4$.

Then the rows of $\jac H$ are dependent over $K$. Furthermore, one of the following
statements hold.
\begin{enumerate}[\upshape (1)]

\item Every set of $6$ rows of $\jac H$ is dependent over $K$.

\item There exists a tame invertible map 
$x + \bar{H}$ such that $\bar{H}$ is quadratic homogeneous and 
$\jac \bar{H} = \jac H$.

\end{enumerate}
In addition, $x + H$ is stably tame if $\chr K \neq 2$.
\end{theorem}

\begin{proof}
The case where $\rk \jac H \le 3$ follows from theorem \ref{rk3np}, so assume that
$\rk \jac H = 4$. We follow the cases of corollary \ref{rk4}.
\begin{itemize}

\item $H$ is as in (1) of corollary \ref{rk4}.

Let $\tilde{H} = S^{-1}H(Sx)$. Then only the first $5$ rows of $\jac \tilde{H}$
may be nonzero. Hence every set of $6$ rows of $\jac H$ is dependent over $K$.

If $n > 5$, then the sixth row of $\jac \tilde{H}$ is zero.
So assume that $n \le 5$. Then it follows from (i) of lemma \ref{lem6} below that
the rows of $\jac H$ are dependent over $K$. 
From \cite[Theorem 4.14]{MR2927368} and theorem \ref{dim5}, it follows that 
$x + \tilde{H}$ is stably tame if $\chr K \neq 2$.

\item $H$ is as in (2) of corollary \ref{rk4}.

Let $\tilde{H} = T^{-1}H(Tx)$ and let $N$ be the leading principal minor matrix
of size $3$ of $\jac \tilde{H}$.
Just like for the case where $H$ is as in (2) of theorem \ref{rk3} in the proof 
of theorem \ref{rk3np}, we can reduce to the case where only the first row
and the first three columns of $\jac \tilde{H}$ may be nonzero. 

From proposition \ref{evenrk}, it follows that we may assume that 
$\hess (\tilde{H}_1|_{x_1=0})$ is a product of a permutation matrix 
and a diagonal matrix. We distinguish three subcases:
\begin{compactitem}

\item $\parder{}{x_2} \tilde{H}_1, \parder{}{x_3} \tilde{H}_1 \in K[x_1,x_2,x_3]$.

Since $\tr \jac \tilde{H} = 0$, we deduce that $\parder{}{x_1} \tilde{H}_1 \in 
K[x_1,x_2,x_3]$ as well. We may assume that the coefficient of either $x_4^2$ 
or $x_4 x_5$ of $\tilde{H}_1$ is nonzero. With techniques in the proof of 
corollary \ref{symdiag}, $x_4x_5$ can be transformed to $x_4^2 - x_5^2$,
so we may assume that the coefficient of $x_4^2$ of $\tilde{H}_1$ is nonzero.

Let $G := \tilde{H}(x_1,x_2,x_3,x_4,0,0,\ldots,0)$. Then 
$\jac G = (\jac \tilde{H})|_{x_5=\cdots=x_n=0}$. Consequently,
the nonzero part of $\jac G$ is restricted to the first four columns.
From (i) of lemma \ref{lem6} below, it follows that the rows of 
$\jac G$ are dependent over $K$ and that $x+G$ is tame.

Since the first row of $\jac G$ is independent of the other rows of $\jac G$, 
the rows of $\jac \tilde{H}$ are dependent over $K$. From proposition 
\ref{tameZ}, it follows that $x + \tilde{H}$ is tame if $\chr K = 0$,
which gives (2) if $\chr K = 0$. Using techniques in the proof of 
theorem \ref{dim5}, the case $\chr K \neq 0$ follows as well.

\item $\parder{}{x_2} \tilde{H}_1, \parder{}{x_3} \tilde{H}_1 \notin 
K[x_1,x_2,x_3]$.

Then we may assume that the coefficients of $x_2 x_4$ and $x_3 x_5$ of
$\tilde{H}_1$ are nonzero. Looking at the coefficients of $x_4^1$ and 
$x_5^1$ of the sum of the determinants of the principal minors of size $2$, 
we infer that $\parder{}{x_1} \tilde{H}_2 = 0$ and 
$\parder{}{x_1} \tilde{H}_3 = 0$.

From a permuted version of \cite[Lemma 2]{1601.00579}, we deduce that
$\jac_{x_2,x_3} \allowbreak (\tilde{H}_2,\tilde{H}_3)$ is nilpotent. 
On account of theorem \ref{rk3np} or \cite[Theorem 3.2]{1601.00579}, 
the second and the third row of $\jac \tilde{H}$ are dependent over $K$.

By applying \cite[Lemma 2]{1601.00579} twice, we see that we can 
replace $\tilde{H}_1$ by $0$ without affecting the nilpotency 
of $\jac \tilde{H}$. Now (2) follows in a similar manner as in the 
case $\parder{}{x_2} \tilde{H}_1, \parder{}{x_3} \tilde{H}_1 \in 
K[x_1,x_2,x_3]$ above.

\item None of the above. 

Then we may assume that $\parder{}{x_2} \tilde{H}_1 \notin 
K[x_1,x_2,x_3]$ and $\parder{}{x_3} \tilde{H}_1 \in K[x_1,x_2,x_3]$.
Furthermore, we may assume that the coefficient of $x_2 x_4$ of 
$\tilde{H}_1$ is nonzero. Now we can apply the same arguments as in the 
case $\parder{}{x_2} \tilde{H}_1, \parder{}{x_3} \tilde{H}_1 \in 
K[x_1,x_2,x_3]$ above.

\end{compactitem}

\item $H$ is as in (3) of corollary \ref{rk4}.

Let $\tilde{H} = T^{-1}H(Tx)$ and let $N$ be the leading principal minor matrix
of size $4$ of $\jac \tilde{H}$.
Just like for the case where $H$ is as in (3) of theorem \ref{rk3} in the proof 
of theorem \ref{rk3np}, we can reduce to the case where only the first row
and the first four columns of $\jac \tilde{H}$ may be nonzero. 

From proposition \ref{evenrk}, it follows that we may assume that 
$\hess (\tilde{H}_1|_{x_1=0})$ is a product of a permutation matrix 
and a diagonal matrix. We distinguish three subcases:
\begin{compactitem}
 
\item $\parder{}{x_2} \tilde{H}_1, \parder{}{x_3} \tilde{H}_1,
\parder{}{x_4} \tilde{H}_1 \in K[x_1,x_2,x_3,x_4]$.

Since $\tr \jac \tilde{H} = 0$, we deduce that $\parder{}{x_1} \tilde{H}_1 \in 
K[x_1,x_2,x_3,x_4]$ as well. We may assume that the coefficient of $x_5 x_6$ of 
$\tilde{H}_1$ is nonzero. 

Just like we reduced to the case where only 
the first $4$ columns of $\jac \tilde{H}$ may be nonzero above, we can reduce
to the case where only the first $6$ columns of $\jac \tilde{H}$ may be nonzero.
Hence the results follow from (ii) of lemma \ref{dim6}.

\item $\parder{}{x_2} \tilde{H}_1, \parder{}{x_3} \tilde{H}_1,
\parder{}{x_4} \tilde{H}_1 \notin K[x_1,x_2,x_3,x_4]$.

Then we may assume that the coefficients of $x_2 x_5$, $x_3 x_6$ and $x_4 x_7$ of
$\tilde{H}_1$ are nonzero. Looking at the coefficients of $x_5^1$, $x_6^1$ and 
$x_7^1$ of the sum of the determinants of the principal minors of size $2$, 
we infer that $\parder{}{x_1} \tilde{H}_2 = 0$, $\parder{}{x_1} \tilde{H}_3 = 0$ 
and $\parder{}{x_1} \tilde{H}_4 = 0$.

From a permuted version of \cite[Lemma 2]{1601.00579}, we deduce that
$\jac_{x_2,x_3,x_4} \allowbreak (\tilde{H}_2,\tilde{H}_3,\tilde{H}_4)$ is nilpotent. 
On account of theorem \ref{rk3np} or \cite[Theorem 3.2]{1601.00579}, 
the second, the third and the fourth row of $\jac \tilde{H}$ are 
dependent over $K$.

By applying \cite[Lemma 2]{1601.00579} twice, we see that we can 
replace $\tilde{H}_1$ by $0$ without affecting the nilpotency 
of $\jac \tilde{H}$. Now (2) follows in a similar manner as in the 
case $\parder{}{x_2} \tilde{H}_1, \parder{}{x_3} \tilde{H}_1,
\parder{}{x_4} \tilde{H}_1 \in K[x_1,x_2,x_3,x_4]$ above.

\item None of the above. 

Then we may assume that $\parder{}{x_2} \tilde{H}_1 \notin 
K[x_1,x_2,x_3,x_4]$ and $\parder{}{x_4} \tilde{H}_1 \in 
K[x_1,x_2,x_3,x_4]$. Furthermore, we may assume that the 
coefficient of $x_2 x_5$ of $\tilde{H}_1$ is nonzero. 

Suppose first that $\parder{}{x_3} \tilde{H}_1 \in K[x_1,x_2,x_3,x_4]$. 
Then we can apply the same argument as in the case $\parder{}{x_2} \tilde{H}_1, 
\parder{}{x_3} \tilde{H}_1, \parder{}{x_4} \tilde{H}_1 \in 
K[x_1,x_2,x_3,x_4]$ above, to reduce to the case where only the 
first $5$ columns of $\jac \tilde{H}$ may be nonzero, which
follows from (ii) of lemma \ref{lem6} below.

Suppose next that $\parder{}{x_3} \tilde{H}_1 \notin K[x_1,x_2,x_3,x_4]$.
Then we may assume that the coefficient of $x_3 x_6$ of $\tilde{H}_1$ 
is nonzero. Now we can apply the same arguments as in the case 
$\parder{}{x_2} \tilde{H}_1, \parder{}{x_3} \tilde{H}_1, \allowbreak
\parder{}{x_4} \tilde{H}_1 \in K[x_1,x_2,x_3]$ above, 
to reduce to the case where only the first $6$ columns of 
$\jac \tilde{H}$ may be nonzero, which
follows from (ii) of lemma \ref{lem6} below. 

\end{compactitem}

\item $H$ is as in (4) of corollary \ref{rk4}.

Let $\tilde{H} = T^{-1}H(Tx)$. Then only the first $5$ columns of 
$\jac \tilde{H}$ may be nonzero, so the results follow from 
(i) of lemma \ref{lem6} below. 

\item $H$ is as in (5) of corollary \ref{rk4}.

Let $\tilde{H} = T^{-1}H(Tx)$. Then only the first $6$ columns of 
$\jac \tilde{H}$ may be nonzero, so the results follow from 
(ii) of lemma \ref{lem6} below. \qedhere

\end{itemize}
\end{proof}

\begin{lemma} \label{lem6}
Let $H$ be quadratic homogeneous, such that $\jac H$ is nilpotent.
\begin{enumerate}[\upshape (i)]

\item If $\chr K \neq 2$ and only the first $5$ columns of $\jac H$
may be nonzero, then the first $5$ rows of $\jac H$ are dependent over $K$, 
and $x + H$ is tame.

\item If $\chr K = 2$ and only the first $6$ columns of $\jac H$
may be nonzero, then the first $6$ rows of $\jac H$ are dependent over $K$, 
and there exists a tame invertible map $x + \bar{H}$ such that 
$\bar{H}$ is quadratic homogeneous and $\jac \bar{H} = \jac H$.

\end{enumerate}
\end{lemma}

\begin{proof}
Let $N$ be the leading principal minor matrix of size $5$ of
$\jac H$ in case of (i), and leading principal minor matrix of size $6$ of
$\jac H$ in case of (ii). From \cite[Lemma 2]{1601.00579}, it follows that $N$
is nilpotent.

If $\rk N \le 3$, then the results follow from theorem \ref{rk3np}.
If $\rk N \ge 4$, then the results follow from theorem \ref{dim5} in case of (i)
and from theorem \ref{dim6} in case of (ii).
\end{proof}

Notice that we only used the case where only the first $4$ columns of $\jac H$
may be nonzero of (i) of lemma \ref{lem6}, which does not require the calculations of
theorem \ref{dim5} to be proved.


\bibliographystyle{quadhmgrk3j}
\bibliography{quadhmgrk3j}

\begin{thebibliography}{BvdEW}

\bibitem[BvdEW]{MR2927368}
Joost Berson, Arno van~den Essen, and David Wright.
\newblock Stable tameness of two-dimensional polynomial automorphisms over a
  regular ring.
\newblock {\em Adv. Math.}, 230(4-6):2176--2197, 2012.

\bibitem[dB1]{1310.7843}
Michiel de~Bondt.
\newblock Mathieu subspaces of codimension less than $n$ of {Mat}$_n({K})$.
\newblock arXiv:1310.7843, 2013.

\bibitem[dB2]{1501.06046}
Michiel de~Bondt.
\newblock Rational maps {$H$} for which {$K(tH)$} has transcendence degree 2
  over {$K$}.
\newblock arXiv:1501.06046, 2015.

\bibitem[dB3]{1609.09753}
Michiel de~Bondt.
\newblock Computations of keller maps over fields with $\frac16$.
\newblock arXiv:1601.09753, 2016.

\bibitem[dB4]{1601.00579}
Michiel de~Bondt.
\newblock Quadratic polynomial maps with jacobian rank two.
\newblock arXiv:1601.00579, 2016.

\bibitem[Dru]{MR1210415}
Ludwik~M. Dru\.zkowski.
\newblock The {J}acobian conjecture in case of rank or corank less than three.
\newblock {\em J. Pure Appl. Algebra}, 85(3):233--244, 1993.

\bibitem[FMN]{DBLP:conf/mfcs/2016}
Piotr Faliszewski, Anca Muscholl, and Rolf Niedermeier, editors.
\newblock {\em 41st International Symposium on Mathematical Foundations of
  Computer Science, {MFCS} 2016, August 22-26, 2016 - Krak{\'{o}}w, Poland},
  volume~58 of {\em LIPIcs}. Schloss Dagstuhl - Leibniz-Zentrum fuer
  Informatik, 2016.

\bibitem[PSS]{DBLP:conf/mfcs/PandeySS16}
Anurag Pandey, Nitin Saxena, and Amit Sinhababu.
\newblock Algebraic independence over positive characteristic: New criterion
  and applications to locally low algebraic rank circuits.
\newblock In Faliszewski et~al. \cite{DBLP:conf/mfcs/2016}, pages 74:1--74:15.

\end{thebibliography}

\end{document}